\documentclass[11pt]{amsart}
\usepackage{amsmath, amssymb, amsthm}
\usepackage[margin=1in]{geometry}
\usepackage[dvipdfmx]{hyperref}
\allowdisplaybreaks

  \makeatletter
    
    \@addtoreset{equation}{section}
  \makeatother
  
\newcommand{\supp}{\mathrm{supp}\hspace{0.5mm}}

\newcommand{\rn}{(\mathbb{R}^n)}
\newcommand{\IF}{\textrm{if}\quad}
\newcommand{\ow}{\textrm{otherwise}}

\newtheorem{theorem}{Theorem}[section]
\newtheorem{lemma}{Lemma}[section]

\newtheorem{proposition}{Proposition}[section]
\newtheorem*{theoremA}{Theorem A}

\theoremstyle{definition}
\newtheorem{remark}{Remark}[section]

\def\R{{\mathbb R}}
\def\C{{\mathbb C}}

\def\L2tx{{L^2(\R_t\times\R^n_x)}}

\def\n#1{{\left\|{#1}\right\|}}

\def\supp{\operatorname{supp}}

\begin{document}

\title[Nonlinear operations on a class of modulation spaces]
{Nonlinear operations on a class of modulation spaces}

\author[T. Kato]{Tomoya Kato}
\author[M. Sugimoto]{Mitsuru Sugimoto}
\author[N. Tomita]{Naohito Tomita}

\address[T. Kato and N. Tomita]{Department of Mathematics, Graduate School of Science, Osaka University, Toyonaka, Osaka 560-0043, Japan}
\address[M. Sugimoto]{Graduate School of Mathematics, Nagoya University Furocho, Chikusa-ku, Nagoya 464-8602, Japan}
\email[T. Kato]{t.katou@cr.math.sci.osaka-u.ac.jp}
\email[M. Sugimoto]{sugimoto@math.nagoya-u.ac.jp}
\email[N. Tomita]{tomita@math.sci.osaka-u.ac.jp}

\keywords{modulation spaces, paradifferential operators}
\subjclass[2010]{42B35, 35S50}

\begin{abstract}
We discuss when the nonlinear operation $f\mapsto F(f)$
maps the modulation space $M^{p,q}_s(\R^n)$ ($1 \leq p,q \leq \infty$)
to the same space again.
It is known that $M^{p,q}_s(\R^n)$ is a multiplication algebra
when $s > n-n/q$, hence it is true for this space if $F$ is entire.
We claim that it is still true for non-analytic $F$ when $q\geq4/3$.
\end{abstract}

\maketitle

\section{Introduction} \label{sec1}

We discuss nonlinear operations $f\mapsto F(f)$, that is, the  
composition of functions $F$ and $f$.
Let $X$ be a function space.
Then does the nonlinear operation map $X$ to the same space $X$?
For the simplest case $F(z)=z^2$, that is, $F(f)=f^2$,
the answer is yes when $X$ is a multiplication algebra.
From this observation, we immediately obtain the affirmative
answer to this question for any entire functions $F(z)$ and 
multiplication algebras $X$.
The typical examples of multiplication algebras are
$L^p$-Sobolev spaces $H^p_s(\R^n)$ ($1<p<\infty$) with $s>n/p$
and Besov spaces $B^{p,q}_s(\R^n)$ ($1\leq p,q\leq\infty$) with $s>n/p$
(see Propositions \ref{alge sob} and \ref{alge bes}).

When $F$ fails to satisfy the analyticity, answering this question
is not so straightforward. 
We however have an affirmative answer by virtue of the theory of
paradifferential operators introduced by Bony \cite{Bo} and developed
by Meyer \cite{meyer 1981}.
The main argument is to write the composition $F(f)$ 
in the form of a linear operation
\[
F(f)=M_{F,f}(x,D)f
\]
(assuming that $F(0)=0$, $f\in H^p_{s}(\R^n)$ is real-valued, and $s>n/p$ to be embedded
in $L^\infty(\R^n)$),
where $M_{F,f}(x,D)$ is a
pseudo-differential operator 
of the H\"ormander class $S^0_{1,1}$.
Since pseudo-differential operators of this class are $H_s^p$-bounded for
$s>0$, we get the following result:

\begin{theoremA}
[{\cite[Theorem 1]{meyer 1981}}]
Let $1 < p < \infty$ and $s > n/p$. 
Assume that $f:\mathbb{R}^n \to \mathbb{R}$, $f \in H^p_s (\mathbb{R}^n) $, 
$F \in C^\infty (\mathbb{R})$ and $F(0) = 0$. 
Then, we have $F(f) \in H^p_s (\mathbb{R}^n) $.
\end{theoremA}

\begin{remark} 
\label{explicit meyer}
We can state Theorem A in an explicit form (see, e.g., Taylor \cite[Section 3.1]{taylor 1991}):
\begin{align*}
\left\| F(f) \right\|_{H^p_s}
\leq
C
\| F^\prime \|_{C^{[s]+1}(\Omega)}
\left(
1
+
\left\| f \right\|_{L^\infty}^{[s]+1}
\right)
\left\| f \right\|_{H^p_s}
,
\end{align*}
where
$\Omega = \{ t : |t| \leq C^\prime \| f \|_{L^\infty} \}$,
and the constants $C$ and $C^\prime$ are universal.
\end{remark}

By the same argument, we have a similar conclusion for Besov spaces
(Runst \cite{runst 1985}),
and a result for complex-valued functions
$f:\R^n\to\C$ can be also stated considering the nonlinear operation
$
f\mapsto F(\operatorname{Re} f, \operatorname{Im} f)
$
with two-variable functions $F(s,t)$,
although in this paper we only consider real-valued functions $f:\R^n\to\R$
just for the sake of simplicity.

The objective of this paper is to establish a similar result
for modulation spaces $M^{p,q}_s(\R^n)$.
Modulation spaces are relatively new function spaces
introduced by Feichtinger \cite{feichtinger 1983} in 1980's 
to measure the decaying and regularity properties
of a function or distribution in a
way different from $L^p$-Sobolev spaces or Besov spaces.
The main idea of modulation spaces is to
consider the space variable and the
variable of its Fourier transform simultaneously,
while they are treated independently in $L^p$-Sobolev spaces and Besov spaces.
Because of their nature, modulation spaces have several significant properties.
For example, the Sch\"odinger propagator $e^{it|D|^2}$ and the
wave propagator $e^{it|D|}$ map the modulation space $M_s^{p,q}$ to
the same space (B\'{e}nyi-Gr\"{o}chenig-Okoudjou-Rogers \cite{BGOR}),
which means, we have no loss of regularity when we work on modulation spaces,
while it is not true for
$L^p$-Sobolev spaces $H^p_s$ or Besov spaces $B_s^{p,q}$
(Miyachi \cite{miyachi 1981}).
When we try to utilize this advantage for nonlinear analysis,
it is indispensable to ask whether
the nonlinear operation also maps $M_s^{p,q}$ to itself.

We know that modulation spaces
$M^{p,q}_s(\R^n)$ ($1 \leq p,q \leq \infty$)
with $s > n/q^\prime$ (with $s=0$ when $q=1$)
are multiplication algebras, where $1/q+1/q'=1$
(Proposition \ref{alge mod}), hence
nonlinear operation $f\mapsto F(f)$
maps these spaces to themselves when $F(z)$ is entire.
Then it is natural to expect the same conclusion for non-analytic $F$
as is the case for $L^p$-Sobolev spaces and Besov spaces.
Unfortunately, it is not obvious because the argument of paradifferential
operators does not work in this case
because pseudo-differential operators
of class $S^0_{1,\delta}$ with $\delta>0$ have exotic mapping property
and are not $M_s^{p,q}$-bounded (see \cite{sugimoto tomita 2008}).
Furthermore, if $F(z)$ is not necessarily analytic, 
a negative answer for $M_0^{p,1}$ is known.
In fact, Bhimani-Ratnakumar \cite{bhimani ratnakumar 2016}
established
that the nonlinear operation
$
f\mapsto F(\operatorname{Re} f, \operatorname{Im} f)
$
is a mapping on $M_0^{1,1}(\R^n)$ if and only
 if $F$ is real analytic and $F(0,0)=0$,
and Kobayashi-Sato \cite{kobayashi sato 2017} generalized 
this result to the case $M_0^{p,1}$ with $1\leq p<\infty$
although it is restricted to the case when $n=1$.
On the other hand, it is still possible for general
$M^{p,q}_s(\R^n)$ with $1<q\leq\infty$ and $s>n/q'$ when $F$ is not analytic.
Our main result states that it is affirmative for $q$ in a range away from
$q=1$:

\begin{theorem}
\label{thm on mod}
Let $1 \leq p < \infty$, $4/3 \leq q < \infty$ (or $p = q = \infty$) and $s > n/q^\prime$.
Assume that $f:\mathbb{R}^n \to \mathbb{R}$,
$f \in M^{p,q}_s (\mathbb{R}^n) $, 
$F \in C^\infty (\mathbb{R})$ and $F(0) = 0$. 
Then, we have $F(f) \in M^{p,q}_s (\mathbb{R}^n) $.
\end{theorem}

We remark that the condition $s>n/q'$ (with $s=0$ when $q=1$) is necessary for modulation spaces
$M_s^{p,q}(\R^n)$ to be multiplication algebras (Guo-Fan-Wu-Zhao \cite{guo fan wu zhao 2017}).
See also Appendix B.
We also remark that 
Theorem \ref{thm on mod} is reduced to the following result
due to the local equivalence between
the modulation spaces $M^{p,q}_s$ and the Fourier Lebesgue spaces
$\mathcal FL^q_s$:

\begin{theorem}
\label{thm on fou}
Let $4/3 \leq q \leq \infty$ and $s > n/q^\prime$. 
Assume that $f:\mathbb{R}^n \to \mathbb{R}$,
 $f \in \mathcal{F} L^q_s (\mathbb{R}^n) $, 
$F \in C^\infty (\mathbb{R})$ and $F(0) = 0$. 
Then, we have $F(f) \in \mathcal{F} L^q_s (\mathbb{R}^n) $.
\end{theorem}

Finally, we just refer to the work by Reich-Reissig-Sickel
\cite{reich reissig Sickel 2016} which also discusses the non-analytic
nonlinear operations, but on modulation spaces with quasi-analytic
regularity.

We explain the organization of this note.
In Section \ref{sec2}, we introduce basic notations of function spaces
and their properties which are used in this paper.
In Section \ref{sec3}, we list examples of multiplication algebras
as a starting point of our argument.
In Section \ref{sec4}, we prove the theorem of nonlinear operation
on Fourier Lebesgue spaces (Theorem \ref{thm on fou}).
In Section \ref{sec5}, we lift it to the case of modulation spaces
(Theorem \ref{thm on mod}) by using the local equivalence between
Fourier Lebesgue spaces and modulation spaces.
This equivalence is a well known fact but the proof is given
in Appendix A for the sake of self-containedness.
In Appendix B, necessity for Fourier Lebesgue spaces and modulation spaces
to be multiplication algebras is considered.

\section{Preliminaries}\label{sec2}

\subsection{Basic notations} \label{sec21}

We collect notations which will be used throughout this paper.
We denote by $\mathbb{R}$, $\mathbb{Z}$ and $\mathbb{Z}_+$
the sets of reals, integers and non-negative integers, respectively. 
The notation $a \lesssim b$ means $a \leq C b$ with a constant $C > 0$ 
which may be different in each occasion, 
and $a \sim b $ means $a \lesssim b$ and $b \lesssim a$. 
For $1 \leq p \leq \infty$, $p^\prime$ is the dual number of $p$ and satisfies that $1/p + 1/p^\prime =1$.
We write $\langle x \rangle = (1 + | x |^2)^{1/2}$ for $x \in \mathbb{R}^n$ 
and $[s] = \max\{ n \in \mathbb{Z} : n \leq s \}$ for $s \in \mathbb{R}$.

We denote the Schwartz space of rapidly decreasing smooth functions
on $\mathbb{R}^n$ by $\mathcal{S} = \mathcal{S} (\mathbb{R}^n)$ and its dual,
the space of tempered distributions, by $\mathcal{S}^\prime = \mathcal{S}^\prime(\mathbb{R}^n)$. 
The Fourier transform and the inverse Fourier transform of $f \in \mathcal{S}(\mathbb{R}^n)$ are given by
\begin{equation*}
\mathcal{F} f  (\xi) = \widehat {f} (\xi) = \int_{\mathbb{R}^n}  e^{-i \xi \cdot x } f(x) d x
\quad\textrm{and}\quad
\mathcal{F}^{-1} f (x) = \check f (x)= \frac{1}{(2\pi)^n} \int_{\mathbb{R}^n}  e^{i x \cdot \xi } f( \xi ) d\xi,
\end{equation*}
respectively.
For $m \in \mathcal{S}^\prime (\mathbb{R}^n)$, 
the Fourier multiplier operator is given by
\begin{equation*}
m(D) f 
=
\mathcal{F}^{-1} \left[m \cdot \mathcal{F} f \right]
=
\left( \mathcal{F}^{-1} m \right) \ast f,
\end{equation*}
and for $s \in \mathbb{R}$ the Bessel potential by
$(I - \Delta )^{s/2} f = \mathcal{F}^{-1} [ \langle \cdot \rangle^s \cdot \mathcal{F}f ]$ 
for $f \in \mathcal{S}(\mathbb{R}^n)$.

We will use some function spaces.
The space of smooth functions with compact support on $\mathbb{R}^n$ 
is denoted by $C^\infty_0 = C^\infty_0 (\mathbb{R}^n)$.
The Lebesgue space $L^p = L^p(\mathbb{R}^n) $ is equipped with the norm 
\begin{equation*}
\| f \|_{L^p} = \left( \int_{\mathbb{R}^n} \big| f(x) \big|^p dx \right)^{1/p} 
\end{equation*}
for $1 \leq p < \infty$. If $p = \infty$, $\| f \|_{L^\infty} = \textrm{ess}\sup_{x\in\mathbb{R}^n} |f(x)|$. 
Moreover, we denote the $L^p$-Sobolev space $H^p_s$
by 
$H^p_s (\mathbb{R}^n) 
= \{ f \in \mathcal{S}^\prime (\mathbb{R}^n) 
: \| f \|_{H^p_s} = \| (I-\Delta)^{s/2} f \|_{L^p} < \infty \}$
for $1 < p < \infty$ and $s\in\mathbb{R}$,
and the (weighted) Fourier Lebesgue space $\mathcal{F} L^p_s$ by
$\mathcal{F} L^p_s (\mathbb{R}^n) 
= \{ f \in \mathcal{S}^\prime (\mathbb{R}^n) 
: \left\| f \right\|_{\mathcal{F} L^p_s} = \| \langle \cdot \rangle^s \widehat f \|_{ L^p } < \infty \}$
for $1 \leq p \leq \infty$ and $s\in\mathbb{R}$.
We remark that $H_s^2 = \mathcal{F} L^2_s$.
Moreover, by the H\"older inequality, we have $\mathcal{F}L^q_s \rn \hookrightarrow H^2_{\widetilde s} \rn$
if $2 < q \leq \infty$ and $n (1/2 -1/q) < s - \widetilde s$. 
Note that the second condition is equivalent to $\widetilde s < n/2 + (s - n/q^\prime)$.
From this relation, we immediately see the following.
\begin{proposition}
\label{fou to sob}
Let $2 < q \leq \infty$, $s > n/q^\prime$ and $ n/2 < \widetilde s < n/2 + (s - n/q^\prime)$. 
Then, we have 
\begin{equation*}
\mathcal{F}L^q_s \rn \hookrightarrow H^2_{\widetilde s} \rn \hookrightarrow L^\infty \rn.
\end{equation*}
\end{proposition}

For $1 \leq q \leq \infty$ and $s\in\mathbb{R}$, we denote by $\ell^q_s$ the set of all complex number sequences 
$\{ a_k \}_{ k \in \mathbb{Z}^n }$ such that
\begin{equation*}
 \| \{ a_k \}_{ k\in\mathbb{Z}^n } \|_{ \ell^q_s } = \left( \sum_{ k \in \mathbb{Z}^n } \langle k \rangle^{sq} | a_k |^q \right)^{ 1 / q } < \infty
\end{equation*}
if $q < \infty$, 
and $\| \{ a_k \}_{ k\in\mathbb{Z}^n } \|_{ \ell^\infty_s } = \sup_{k \in \mathbb{Z}^n} \langle k \rangle^s |a_k| < \infty$ if $q = \infty$.
For the sake of simplicity, we will write $ \| a_k  \|_{ \ell^q_s } $
instead of the more correct notation $ \| \{ a_k \}_{ k\in\mathbb{Z}^n } \|_{ \ell^q_s } $.

We end this subsection by mentioning a key fact
on the boundedness of Fourier multiplier operators invented by Hahn \cite[Theorem 9]{hahn 1967}.

\begin{proposition}
\label{hahn}
Let $2 \leq p < \infty$ and $s > n/p$. 
Then, if 
$2p/(p+2) \leq q \leq 2p/(p-2)$ when $p\neq 2$ 
or $1 \leq q < \infty$ when $p = 2$, 
we have
\begin{equation*}
\left\| m ( D ) f \right\|_{ L^q } \lesssim \left\| m \right\|_{ H^p_s } \left\| f \right\|_{ L^q }
\end{equation*}
for all $m \in H^p_s ( \mathbb{R}^n )$ and all $f \in L^q ( \mathbb{R}^n )$.
\end{proposition}

\begin{remark}
\label{rem hahn}
In Proposition \ref{hahn}, we excluded $q = \infty$ for the case $p=2$. 
This comes from that $\mathcal{S}$ is not dense in $L^\infty$.
In this case, we regard $m ( D ) f $ as the convolution $(\mathcal{F}^{-1} m ) \ast f$.
Then, this is well-defined since $H^2_s \rn \hookrightarrow \mathcal{F}L^1_0 \rn$ for $s > n/2$, 
and thus Proposition \ref{hahn} holds for $q = \infty$ and $p=2$.
In fact, if $s > n/2 $,
\begin{equation*}
\left\| m ( D ) f \right\|_{ L^\infty } 
=
\left\| (\mathcal{F}^{-1} m ) \ast f \right\|_{ L^\infty } 
\leq 
\left\| \mathcal{F}^{-1} m \right\|_{ L^1 } \left\| f \right\|_{ L^\infty }
\lesssim 
\left\| m \right\|_{ H^2_s } \left\| f \right\|_{ L^\infty }
\end{equation*}
holds for all $m \in H^2_s \rn$ and all $f \in L^\infty \rn$.
\end{remark}

\subsection{Modulation spaces} \label{sec22}

We give the definition of modulation spaces which were introduced 
by Feichtinger \cite{feichtinger 1983} (see also Gr\"ochenig \cite{grochenig 2001}). 
We fix a function (called a window function) $g \in \mathcal{S} (\mathbb{R}^n) \setminus \{0\}$ 
and denote the short-time Fourier transform of $f \in \mathcal{S}^\prime (\mathbb{R}^n) $ 
with respect to $g$ by 
\begin{equation*}
V_g f (x,\xi) = \int_{\mathbb{R}^n} e^{ - i \xi \cdot t } \ \overline{g (t-x)} \, f(t) d t.
\end{equation*}
We will sometimes write $V_g [ f ] $ when the form of $f$ is complicated. 
For $1 \leq p,q \leq \infty$ and $s \in \mathbb{R}$, the modulation space $M^{ p , q }_s$ is defined by
\begin{equation*}
M^{p,q}_s  (\mathbb{R}^n)  = \left\{ f\in\mathcal{S}^\prime(\mathbb{R}^n) : \left\| f \right\|_{M^{p,q}_s} = \left\| \left\| \langle \xi \rangle^s V_g f (x,\xi) \right\|_{L^p(\mathbb{R}^n_x)} \right\|_{L^q(\mathbb{R}^n_\xi)} < \infty \right\}.
\end{equation*}

We note that the definition of modulation spaces is independent of the choice of window functions.
$M^{p,q}_s$ are Banach spaces
and $\mathcal{S} \subset M^{p,q}_s \subset \mathcal{S}^\prime$. 
In particular, $\mathcal{S}$ is dense in $M^{p,q}_s$ if $1 \leq p,q < \infty$. 
For $1 \leq p,q < \infty$, 
the dual space of $M^{p,q}_s$ can be seen as $\big( M^{p,q}_s \big)^\prime =  M^{p^\prime , q^\prime }_{-s} $.
Moreover, we have the following complex interpolation theorem. 
If $0 < \theta < 1$, $s = (1-\theta)s_1 + \theta s_2$, $ 1/p= (1-\theta)/p_1 + \theta/p_2 $ and $ 1/q= (1-\theta)/q_1 + \theta/q_2$, 
we have $\big( M^{p_1,q_1}_{s_1} , M^{p_2,q_2}_{s_2} \big)_\theta = M^{p,q}_{s}$. 
As a further elementary property, 
we note the following embedding
proved by Feichtinger \cite[Propositoin 6.5]{feichtinger 1983}.

\begin{proposition} \label{basic emb}
Let $1 \leq p_1, p_2, q_1, q_2 \leq \infty$ and $s_1, s_2 \in \mathbb{R}$. 
Then, we have
$M^{p_1,q_1}_{s_1} \hookrightarrow M^{p_2,q_2}_{s_2}$ 
for $p_1 \leq p_2$, $q_1 \leq q_2$ and $s_1 \geq s_2$.
\end{proposition}

\subsection{Besov spaces} \label{sec23}

We here give the definition of Besov spaces (see also \cite[Section 2.3]{triebel 1983}). 
Let $\varphi \in \mathcal{S} (\mathbb{R}^n)$ 
satisfy that $\varphi = 1$ on $\{ \xi : | \xi | \leq 1/2 \}$ 
and $\supp \varphi \subset \{ \xi : | \xi | \leq 1 \}$. 
We put $\psi = \varphi (\cdot/2) - \varphi (\cdot)$, 
and then see that $\supp \psi \subset \{ \xi : 1/2 \leq | \xi | \leq 2 \}$.
Moreover, we set $\varphi_j = \varphi (\cdot / 2^j )$ and $\psi_j = \psi (\cdot / 2^j )$ 
for $j \in \mathbb{Z}_+$,
and denote the Fourier multiplier operators with respect to them by
\begin{align*}
S_j f = \varphi_j (D) f
\quad\textrm{and}\quad
\Delta_j f = \psi_j (D) f.
\end{align*}
We remark that 
\begin{align*}
\varphi + \sum_{ j=0 }^\infty \psi_j = 1 
\quad\textrm{and}\quad
S_0 f + \sum_{ j=0 }^\infty \Delta_j f = f,
\end{align*}
and also that $\Delta_j f = S_{j+1} f - S_j f $.
By using these notations, for $1 \leq p,q \leq \infty$ and $s \in \mathbb{R}$, 
the Besov space $B^{p,q}_s $ is defined by
\begin{equation*}
B^{p,q}_s (\mathbb{R}^n) 
= \left\{ f\in\mathcal{S}^\prime(\mathbb{R}^n) : \left\| f \right\|_{B^{p,q}_s} 
= \left\| S_0 f \right\|_{L^p} 
+ \left( \sum_{j=0}^\infty 2^{jsq} \left\| \Delta_j f \right\|_{L^p}^q \right)^{1/q}
< \infty \right\}.
\end{equation*}
Note that the norm of the Besov space is read with the usual modification for $ q = \infty$.
Besov spaces also have basic properties like modulation spaces, namely,
completeness, density, duality and interpolation. 
However, we omit mentioning the details and refer the reader to \cite[Section 2.3]{triebel 1983}.

\section{Multiplication algebras}
\label{sec3}

In this section, we collect some properties called multiplication algebras.
A function space $X$ is said to be a multiplication algebra 
if for all $f,g \in X$ the product $f \cdot g$ exists and belongs to $X$,
and if the inequality 
$\| f \cdot g \|_X \lesssim \| f \|_X \cdot \| g \|_X$
holds for all $f,g \in X$.
More precisely, see \cite[Section 2.8]{triebel 1983}.
The following results on $L^p$-Sobolev and Besov spaces are well-known
(see, e.g., Strichartz \cite[Chapter II, Theorem 2.1]{strichartz 1967} 
and Triebel \cite[Theorem 2.8.3]{triebel 1983}).

\begin{proposition}
\label{alge sob}
Let $1 < p < \infty$ and $s > n/p$. Then, we have
\begin{equation*}
\| f \cdot g \|_{H^p_s} \lesssim \| f \|_{H^p_s} \cdot \| g \|_{H^p_s}
\end{equation*}
for all $f,g \in H^p_s (\mathbb{R}^n)$.
\end{proposition}

\begin{proposition}
\label{alge bes}
Let $1 \leq p,q \leq \infty$ and $s > n/p$. Then, we have
\begin{equation*}
\| f \cdot g \|_{B^{p,q}_s} \lesssim \| f \|_{B^{p,q}_s} \cdot \| g \|_{B^{p,q}_s}
\end{equation*}
for all $f,g \in B^{p,q}_s (\mathbb{R}^n)$.
\end{proposition}

Some of modulation spaces are also multiplication algebras
(see, e.g, Feichtinger \cite[Remark 6.4 and Proposition 6.9]{feichtinger 1983} 
and Sugimoto-Tomita-Wang \cite[Proposition 3.2]{sugimoto tomita wang 2011}).

\begin{proposition}
\label{alge mod}
Let $1 \leq p,q \leq \infty$ and $s > n/q^\prime$. Then, we have
\begin{equation*}
\| f \cdot g \|_{M^{p,q}_s} \lesssim \| f \|_{M^{p,q}_s} \cdot \| g \|_{M^{p,q}_s}
\end{equation*}
for all $f,g \in M^{p,q}_s (\mathbb{R}^n)$, and
\begin{equation*}
\| f \cdot g \|_{M^{p,1}_0} \lesssim \| f \|_{M^{p,1}_0} \cdot \| g \|_{M^{p,1}_0}
\end{equation*}
for all $f,g \in M^{p,1}_0 (\mathbb{R}^n)$.
\end{proposition}

Finally, we give the following counterpart for Fourier Lebesgue spaces.

\begin{proposition} \label{alge fou}
Let $1 \leq q \leq \infty$ and $s > n / q^\prime$.
Then, we have
\begin{equation*}
\| f \cdot g \|_{ \mathcal{F} L^q_s } \lesssim \| f \|_{ \mathcal{F} L^q_s } \cdot \| g \|_{ \mathcal{F} L^q_s }
\end{equation*}
for all $f,g \in \mathcal{F} L^q_s (\mathbb{R}^n)$, and
\begin{equation*}
\| f \cdot g \|_{ \mathcal{F} L^1_0 } \lesssim \| f \|_{ \mathcal{F} L^1_0 } \cdot \| g \|_{ \mathcal{F} L^1_0 }
\end{equation*}
for all $f,g \in \mathcal{F} L^1_0 (\mathbb{R}^n)$.
\end{proposition}

\begin{proof}[\bf Proof of Proposition \ref{alge fou}] 
From the inequality $\langle \xi \rangle^s \lesssim \langle \xi - \eta \rangle^s + \langle \eta \rangle^s $ 
for any $\xi, \eta \in \mathbb{R}^n$ and $s \geq 0$, we have
\begin{align*} 
\left\| f \cdot g \right\|_{ \mathcal{F} L^q_s } 
&\sim
\left\| \langle \xi \rangle^s \int_{ \mathbb{R}^n } \widehat f (\xi - \eta ) \cdot \widehat g ( \eta ) d \eta \right\|_{ L^q (\mathbb{R}^n_\xi) } 
\\
&\lesssim
\left\| 
	\int_{ \mathbb{R}^n } \langle \xi - \eta \rangle^s 
	\left| \widehat f (\xi - \eta ) \right|
	\cdot 
	\left| \widehat g ( \eta ) \right| 
	d\eta 
\right\|_{ L^q (\mathbb{R}^n_\xi) } 
+
\left\| 
	\int_{ \mathbb{R}^n } 
	\left| \widehat f (\xi - \eta ) \right|
	\cdot 
	\langle \eta \rangle^s \left| \widehat g ( \eta )  \right|
	d \eta 
\right\|_{ L^q (\mathbb{R}^n_\xi) }.
\end{align*}
Then, we have by the Young and H\"older inequalities
\begin{align*} 
\left\| f \cdot g \right\|_{ \mathcal{F} L^q_s } 
&\lesssim
\left\| \langle \cdot \rangle^s \widehat f  \right\|_{L^q} \cdot \left\| \widehat g \right\|_{ L^{ 1 } } 
+
\left\| \widehat f \right\|_{ L^{ 1 } } \cdot \left\| \langle \cdot \rangle^s \widehat g \right\|_{ L^q }.
\\
&\leq
\left\| f \right\|_{ \mathcal{F} L^q_s } 
\cdot \| \langle \cdot \rangle^{-s} \|_{L^{q^\prime}} \left\| g \right\|_{ \mathcal{F} L^q_s }
+
\| \langle \cdot \rangle^{-s} \|_{L^{q^\prime}} \left\| f \right\|_{ \mathcal{F} L^q_s } 
\cdot \left\| g \right\|_{ \mathcal{F} L^q_s },
\end{align*} 
which yields from the assumption $s > n / q^\prime$ that
$
\left\| f \cdot g \right\|_{ \mathcal{F} L^q_s } 
\lesssim
 \| f \|_{ \mathcal{F} L^q_s } \cdot \| g \|_{ \mathcal{F} L^q_s }
$.
Here, we remark that, 
in the case $q=1$,
$
\| \langle \cdot \rangle^{-s} \|_{L^{q^\prime}} 
$
is finite even if $s = 0$,
which gives the conclusion for $q=1$ and $s=0$.
\end{proof}

\section{Proof of Theorem \ref{thm on fou}} \label{sec4}

We begin this section with an observation which will be used
in the proof of Theorem \ref{thm on fou}. 
Put
\begin{align}
\label{exp g}
G(t) = F(t) - \sum_{k = 1}^N F^{ (k) } (0) \frac{ t^k }{ k! }
\end{align}
for any $N \in \mathbb{N}$, where $F \in C^\infty ( \mathbb{R} ) $ and $F (0) = 0$. 
Then, we see that $G (0) = G^{ (1) } (0)= \cdots= G^{ (N) } (0) = 0$, and have 
\begin{align}
\label{f and g}
F(f) = G(f) + \sum_{k = 1}^N F^{ (k) } (0) \frac{ f^k }{ k! }.
\end{align}
In order to obtain Theorem \ref{thm on fou}, 
we will prove that the right hand side of \eqref{f and g} belongs to $\mathcal{F} L^q_s$. 
However, it is trivial that the second term belongs to $\mathcal{F} L^q_s$,
since $\mathcal{F} L^q_s (\mathbb{R}^n)$ with $s > n / q^\prime$ 
is a multiplication algebra (see Proposition \ref{alge fou}).
Hence, Theorem \ref{thm on fou} is reduced to the following statement.

\begin{proposition}
\label{thm on fou reduction}
Let $4/3 \leq q \leq \infty$ and $s > n/q^\prime$.
Assume that $f:\mathbb{R}^n \to \mathbb{R}$,
$f \in \mathcal{F} L^q_s(\mathbb{R}^n)$, 
$G \in C^\infty (\mathbb{R})$ 
and $G (0) = G^{ (1) } (0)= \cdots= G^{ ([s]+2) } (0) = 0$. 
Then, we have $G(f) \in \mathcal{F} L^q_s(\mathbb{R}^n)$.
\end{proposition}

Before starting the proof of Proposition \ref{thm on fou reduction}, 
we transform $G(f)$ to a more manageable alternative expression,
which was provided by Meyer \cite[Section 2]{meyer 1981}.
We first remark 
that $
f \in \mathcal{F} L^q_s (\mathbb{R}^n) 
\hookrightarrow H^{q^\prime}_s (\mathbb{R}^n)$ 
holds for $q \leq 2$,
and that $f \in \mathcal{F}L^q_s (\mathbb{R}^n) 
\hookrightarrow H^2_{\widetilde s} (\mathbb{R}^n) $
for $q > 2$ and $ n/2 < \widetilde s < n/2 + (s - n/q^\prime)$ (see Proposition \ref{fou to sob}).
They imply that $f$ belongs to $B^{\infty,1}_0 (\mathbb{R}^n) $, hence to $L^\infty (\mathbb{R}^n)$,
and so $f$ is a bounded uniformly continuous function.
Then $S_j f$ converges uniformly to $f$ as $j \to \infty$,
and $G(f) = G ( \lim_{j \to \infty} S_j f ) = \lim_{j \to \infty} G ( S_j f ) $. 
By the mean value theorem and the fact $S_{j+1} f = S_j f + \Delta_j f$, we have
\begin{align*}
G(f) 
&= 
G(S_0 f) + \sum_{j = 0}^\infty \left[ G(S_{j+1} f) - G(S_j f) \right]
\\
&= 
G(S_0 f) + \sum_{j = 0}^\infty \int_0^1 G^{(1)} \left( S_j f + t \Delta_j f \right) d t \cdot \Delta_j f
=
G(S_0 f) + \sum_{j = 0}^\infty m_{j} \cdot \Delta_j f 
,
\end{align*}
where we set
\begin{equation}
\label{mk}
m_{j} = \int_0^1 G^{(1)} \left( S_j f + t \Delta_j f \right) d t.
\end{equation}
Moreover, we decompose $m_{j}$ into the low and high frequency parts. 
Recall from Section \ref{sec23}
that $ \varphi (\xi) + \sum_{m = 0}^\infty \psi ( 2^{-m} \xi ) = 1 $
for any $\xi \in \mathbb{R}^n$.
Then, it follows that
\begin{equation*}
\varphi \left( \frac{ \xi }{ C \cdot 2^{j} } \right) + \sum_{m = 0}^\infty \psi \left( \frac{ \xi }{ C \cdot 2^{j+m} } \right) = 1 
\end{equation*}
for any $\xi \in \mathbb{R}^n$, 
where $C$ is a sufficiently large constant. 
Using this decomposition, we have
\begin{align*}
m_{j} 
=
\varphi \left( \frac{ D }{ C \cdot2^{j} } \right) m_{j} 
+ \sum_{m = 0}^\infty \psi \left( \frac{ D }{ C \cdot2^{j+m} } \right) m_{j}
=
q_j + \sum_{m = 0}^\infty p_{j,m},
\end{align*}
where we set 
\begin{equation} \label{qkpkm}
q_j = \varphi \left( \frac{ D }{ C \cdot 2^{j} } \right) m_{j}
\quad\textrm{and}\quad
p_{j,m} = \psi \left( \frac{ D }{ C \cdot 2^{j+m} } \right) m_{j}.
\end{equation}
Therefore, $G(f)$ is expressed in the following form:
\begin{equation}
\label{final decom f}
G(f) = G(S_0 f) 
+ \sum_{j = 0}^\infty q_j \cdot \Delta_j f 
+ \sum_{j = 0}^\infty \sum_{m = 0}^\infty p_{j,m} \cdot \Delta_j f .
\end{equation}

From now on, we give estimates for each terms 
of the expression \eqref{final decom f} without specifying
constants explicitly.
(We however remark that these implicit constants may depend
on $\n{f}_{\mathcal FL^q_s}$.)
We start by stating 
two lemmas.
The first one is for $q_j$ in \eqref{qkpkm}.

\begin{lemma}
\label{lem qk}
Let $1 < q \leq \infty$, $s > n/q^\prime$ and $ n/2 < \widetilde s < n/2 + (s - n/q^\prime)$. 
Suppose that $f \in \mathcal{F} L^q_s (\mathbb{R}^n)$
and all the assumptions of $G$ are the same as in Proposition \ref{thm on fou reduction}. 
Then, we have
\begin{align*}
\left\| q_j \right\|_{ H^{q^\prime}_s } 
&\lesssim 1
\quad\textrm{if}\quad
1 < q \leq 2;
\\
\left\| q_j \right\|_{ H^{2}_{\widetilde s} } 
&\lesssim 1
\quad\textrm{if}\quad
2 < q \leq \infty
\end{align*}
for any $j \in \mathbb{Z}_+$. 
Here, the implicit constants are independent of $j \in \mathbb{Z}_+$.
\end{lemma}

\begin{proof}[\bf Proof of Lemma \ref{lem qk}]
We first consider the estimate with $1 < q \leq 2$.
Set $f_{j,t} = S_j f + t \Delta_j f $. 
Recalling the definition of $m_{j}$ from \eqref{mk}, we have
\begin{align*}
\left\| q_j \right \|_{ H^{q^\prime}_s } 
\lesssim
\left\| m_{j} \right\|_{ H^{q^\prime}_s } 
\leq
\int_0^1 \left\| G^{(1)} \left( f_{j,t} \right) \right\|_{ H^{q^\prime}_s } d t.
\end{align*}
Observe that
\begin{align*}
\left\| f_{j,t} \right\|_{ H^{q^\prime}_s } 
&\lesssim 
\left( \| \mathcal{F}^{-1} \varphi_j \|_{L^1} 
+ t \| \mathcal{F}^{-1} \psi_j \|_{L^1} \right)
\| f \|_{ H^{q^\prime}_s } 
\\
&\lesssim 
\left( \| \mathcal{F}^{-1} \varphi \|_{L^1} 
+ t \| \mathcal{F}^{-1} \psi \|_{L^1} \right)
\| f \|_{ \mathcal {F} L^q_s} 
\lesssim
\| f \|_{ \mathcal {F} L^q_s} ,
\end{align*} 
which means that $f_{j,t} \in H^{q^\prime}_s$ for any $j \in \mathbb{Z}_+$ and any $t \in [0,1]$.
Then, using Theorem A and Remark \ref{explicit meyer} 
together with the assumptions $G \in C^\infty (\mathbb{R})$ and $G^{(1)} (0) = 0$, 
we have
\begin{align*}
\left\| G^{(1)} ( f_{j,t}) \right\|_{H^{q^\prime}_s}
&\lesssim
\| G^{(2)} \|_{C^{[s]+1}(\Omega)}
\left( 1 + \left\| f_{j,t} \right\|_{L^\infty}^{[s]+1}\right)
\left\| f_{j,t} \right\|_{H^{q^\prime}_s}
\\
&\lesssim
\| G \|_{C^{[s]+3}(\Omega)}
\left( 1 + \left\| f \right\|_{L^\infty}^{[s]+1}\right)
\| f \|_{ \mathcal {F} L^q_s},
\end{align*}
where 
$\Omega = \{ t : |t| \lesssim \| f \|_{L^\infty} \}$.
Note that the last quantity is finite
since $f \in \mathcal{F} L^q_s (\mathbb{R}^n) \hookrightarrow L^\infty (\mathbb{R}^n)$ for $s > n/q^\prime$
and the smooth function $G \in C^\infty (\mathbb{R} )$
is measured by $C^{[s]+3} $ on the closed and bounded domain $\Omega$. 
Therefore, we have $ \left\| q_j \right\|_{ H^{q^\prime}_s } \lesssim 1$ for $1 < q \leq 2$.

We next consider the estimate with $2 < q \leq \infty$.
This is, however, immediately given by the same argument as above.
In fact, since we already know from Proposition \ref{fou to sob} 
that $f \in \mathcal{F}L^q_s \rn \hookrightarrow H^2_{\widetilde s} \rn \hookrightarrow L^\infty \rn$,
we have by Theorem A and Remark \ref{explicit meyer}
\begin{align*}
\left\| G^{(1)} ( f_{j,t}) \right\|_{ H^{2}_{\widetilde s} }
&\lesssim
\| G^{(2)} \|_{C^{[\widetilde s]+1}(\Omega)}
\left(
1
+
\left\| f_{j,t} \right\|_{L^\infty}^{[\widetilde s]+1}
\right)
\left\| f_{j,t} \right\|_{ H^{2}_{\widetilde s} }
\\
&\lesssim
\| G \|_{C^{[\widetilde s]+3}(\Omega)}
\left(
1
+
\left\| f \right\|_{L^\infty}^{[\widetilde s]+1}
\right)
\left\| f \right\|_{ H^{2}_{\widetilde s} }
\lesssim
\| G \|_{C^{[\widetilde s]+3}(\Omega)}
\left( 1 + \left\| f \right\|_{L^\infty}^{[\widetilde s]+1}\right)
\| f \|_{ \mathcal {F} L^q_s}.
\end{align*}
Note that the last quantity is finite. 
Hence, we obtain 
$\left\| q_j \right\|_{ H^{2}_{\widetilde s} } \lesssim 1$
for $2 < q \leq \infty$.
\end{proof}

The second one is concerning $p_{j,m}$ in \eqref{qkpkm}.

\begin{lemma}
\label{lem pkm}
Let $1 < q \leq \infty$, $s > n/q^\prime$ and $ n/2 < \widetilde s < n/2 + (s - n/q^\prime)$. 
Suppose that $f \in \mathcal{F} L^q_s (\mathbb{R}^n)$
and all the assumptions of $G$ are the same as in Proposition \ref{thm on fou reduction}. 
Then, we have
\begin{align*}
\left\| p_{j,m} \right\|_{ H^{q^\prime}_s } 
&\lesssim 
2^{-m([s]+1)}
\quad\textrm{if}\quad
1 < q \leq 2;
\\
\left\| p_{j,m} \right\|_{ H^{2}_{\widetilde s} } 
&\lesssim 
2^{-m([s]+1)}
\quad\textrm{if}\quad
2 < q \leq \infty
\end{align*}
for any $j,m \in \mathbb{Z}_+$.
Here, the implicit constants are independent of $j,m \in \mathbb{Z}_+$.
\end{lemma}

To prove Lemma \ref{lem pkm}, we prepare the following:

\begin{lemma}
\label{deriv mj}
Let $1 < q \leq \infty$, $s > n/q^\prime$ and $ n/2 < \widetilde s < n/2 + (s - n/q^\prime)$,
and let $\alpha \in \mathbb{Z}_+^n$ satisfy that $| \alpha | = [s]+1$.
Suppose that $f \in \mathcal{F} L^q_s (\mathbb{R}^n)$
and all the assumptions of $G$ are the same as in Proposition \ref{thm on fou reduction}.
Then, we have
\begin{align*}
\| \partial^\alpha m_{j}  \|_{ H^{q^\prime}_s } & \lesssim 2^{j ([s]+1)} 
\quad\textrm{if}\quad
1 < q \leq 2;
\\
\| \partial^\alpha m_{j}  \|_{ H^{2}_{\widetilde s} } & \lesssim 2^{j ([s]+1)} 
\quad\textrm{if}\quad
2 < q \leq \infty
\end{align*}
for any $j \in \mathbb{Z}_+$.
Here, the implicit constants are independent of $j \in \mathbb{Z}_+$.
\end{lemma}

\begin{proof}[\bf Proof of Lemma \ref{deriv mj}]
We first consider the case $1 < q \leq 2$.
Set $f_{j,t} = S_j f + t \Delta_j f $. Then we have by Proposition \ref{alge sob}
\begin{align*}
\| \partial^\alpha m_{j} \|_{ H^{q^\prime}_s }
&\leq 
\int_0^1 \left\| \partial^\alpha [ G^{(1)} \left( f_{j,t} \right) ] \right\|_{ H^{q^\prime}_s } d t
\\
&\lesssim 
\sum_{ \mu = 1 }^{ |\alpha| } \sum_{\alpha_1 + \cdots + \alpha_\mu = \alpha }
\int_0^1
\left\| G^{ ( \mu +1 ) } ( f_{j,t} ) \right\|_{ H^{q^\prime}_s }
\cdot \left\| \partial^{\alpha_1} f_{j,t} \right\|_{ H^{q^\prime}_s }
\cdots \left\| \partial^{\alpha_\mu}  f_{j,t} \right\|_{ H^{q^\prime}_s }
d t,
\end{align*}
where $|\alpha|=[s]+1$. 
Observe that for $\beta \in \mathbb{Z}_+^n$
\begin{align*}
\left\| \partial^{\beta} f_{j,t} \right\|_{ H^{q^\prime}_s } 
\lesssim
\left( \| \mathcal{F}^{-1} [ \xi^{\beta} \cdot \varphi_j ] \|_{L^1} 
+ t \| \mathcal{F}^{-1} [ \xi^{\beta} \cdot \psi_j ] \|_{L^1} \right)
\| f \|_{ H^{q^\prime}_s }
\lesssim
2^{ j | \beta | } \| f \|_{ \mathcal{F} L^q_s },
\end{align*}
which also means that $ f_{j,t} \in H^{q^\prime}_s $ for any $j \in \mathbb{Z}_+$ and any $t \in [0,1]$. 
Therefore, by using Theorem A and Remark \ref{explicit meyer} 
together with the assumptions $G \in C^\infty(\mathbb{R})$ and $ G^{ (2) } (0)= \cdots= G^{ ([s]+2) } (0) = 0$,
we have for $\mu = 1, \cdots , [s]+1$
\begin{align*}
\left\| G^{ ( \mu +1 )} (f_{j,t}) \right\|_{H^{q^\prime}_s}
&\lesssim
\| G^{ ( \mu +2 )} \|_{C^{[s]+1} (\Omega)} 
\left(
1
+
\left\| f_{j,t} \right\|_{L^\infty}^{[s]+1}
\right)
\left\| f_{j,t} \right\|_{H^{q^\prime}_s}
\\
&\lesssim
\| G \|_{C^{\mu+[s]+3} (\Omega)} 
\left( 1 + \| f \|_{L^\infty}^{[s]+1}\right)
\left\| f \right\|_{ \mathcal{F} L^q_s },
\end{align*}
where
$\Omega = \{ t : |t| \lesssim \| f \|_{L^\infty} \}$.
Note that the last quantity makes sense surely 
since $f \in \mathcal{F} L^q_s (\mathbb{R}^n) \hookrightarrow L^\infty (\mathbb{R}^n)$ for $s > n/q^\prime$
and $G \in C^\infty (\mathbb{R} )$ is considered on the closed and bounded domain $\Omega$. 
Hence, we obtain
\begin{align*}
\| \partial^\alpha m_{j} \|_{ H^{q^\prime}_s }
&\lesssim
\sum_{ \mu = 1 }^{ [s]+1 } \sum_{\alpha_1 + \cdots + \alpha_\mu = \alpha }
\left( 2^{ j | \alpha_1 | }  \| f \|_{ \mathcal{F} L^q_s } \right) 
\cdots 
\left( 2^{ j | \alpha_\mu | } \| f \|_{ \mathcal{F} L^q_s } \right)
\lesssim
2^{ j ([s]+1) }
,
\end{align*}
which completes the proof for $1 < q \leq 2$.

We next consider the case $2 < q \leq \infty$. 
Repeating the same lines as above, 
since we already know from Proposition \ref{fou to sob} 
that $f \in \mathcal{F}L^q_s \rn \hookrightarrow H^2_{\widetilde s} \rn \hookrightarrow L^\infty \rn$,
we have for $\beta \in \mathbb{Z}_+^n$
\begin{align*}
\left\| \partial^{\beta} f_{j,t} \right\|_{ H^2_{\widetilde s} } 
\lesssim
\left( \| \mathcal{F}^{-1} [ \xi^{\beta} \cdot \varphi_j ] \|_{L^1} 
+ t \| \mathcal{F}^{-1} [ \xi^{\beta} \cdot \psi_j ] \|_{L^1} \right)
\| f \|_{ H^2_{\widetilde s} }
\lesssim
2^{ j | \beta | } \| f \|_{ \mathcal{F} L^q_s },
\end{align*}
and by Theorem A and Remark \ref{explicit meyer} for $\mu = 1, \cdots , [s]+1$
\begin{align*}
\left\| G^{ ( \mu +1 )} (f_{j,t}) \right\|_{ H^2_{\widetilde s} }
&\lesssim
\| G^{ ( \mu +2 )} \|_{C^{[\widetilde s]+1} (\Omega)} 
\left(
1
+
\left\| f_{j,t} \right\|_{L^\infty}^{[\widetilde s]+1}
\right)
\left\| f_{j,t} \right\|_{ H^2_{\widetilde s} }
\\
&\lesssim
\| G \|_{C^{\mu+[\widetilde s]+3} (\Omega)} 
\left( 1 + \| f \|_{L^\infty}^{[\widetilde s]+1}\right)
\left\| f \right\|_{ \mathcal{F} L^q_s }.
\end{align*}
Hence, we obtain 
$
\| \partial^\alpha m_{j} \|_{ H^2_{\widetilde s} }
\lesssim
2^{ j ([s]+1) }
$
for $2 < q \leq \infty$.
\end{proof}

\begin{proof}[\bf Proof of Lemma \ref{lem pkm}]
By the moment condition of $\psi$ and a Taylor expansion, we have
\begin{align*}
p_{j,m} (x)
&= 
C^n \cdot 2^{ ( j+m)n } 
\int_{ \mathbb{R}^n }
\check \psi (C \cdot 2^{j+m} y ) \cdot m_{j} ( x - y ) d y
\\
&=
C^n \cdot 2^{ ( j+m)n } 
\int_{ \mathbb{R}^n } \check \psi (C \cdot 2^{j+m} y ) \left\{ m_{j} ( x - y ) 
- \sum_{|\alpha| < M} \frac{\left( - y \right)^\alpha}{\alpha !} ( \partial^\alpha m_{j} ) (x) \right\} d y
\\
&=
C^n \cdot 2^{ ( j+m)n } 
\int_{ \mathbb{R}^n } \check \psi (C \cdot 2^{j+m} y ) \cdot 
\left\{ 
M\sum_{|\alpha| = M } \frac{\left( - y \right)^\alpha}{\alpha !} 
\int_0^1 (1-t)^{M-1} \cdot ( \partial^\alpha m_{j} ) (x-ty) d t 
\right\} d y,
\end{align*}
where $M = [s]+1$. 
Taking the $H^{q^\prime}_s$-norm to the both sides, we have
\begin{align*}
\| p_{j,m} \|_{ H^{q^\prime}_s }
&\lesssim
2^{ ( j+m ) n } \int_{ \mathbb{R}^n } | \check \psi (C \cdot 2^{j+m} y ) | \cdot 
| y |^{[s]+1} 
\left\{ 
\sum_{|\alpha| = [s]+1 } \int_0^1 \| ( \partial^\alpha m_{j} ) (x-ty) \|_{ H^{q^\prime}_s (\mathbb{R}^n_x) } d t
\right\} d y
\\
&\sim
2^{ - ( j+m )( [s]+1 ) } \left( \int_{ \mathbb{R}^n } | \check \psi ( y ) | \cdot | y |^{[s]+1} d y \right) \sum_{|\alpha| = [s]+1 } \| \partial^\alpha m_{j} \|_{ H^{q^\prime}_s }
\\
&\sim
2^{ - ( j+m )( [s]+1) } \sum_{|\alpha| = [s]+1 } \| \partial^\alpha m_{j}  \|_{ H^{q^\prime}_s }.
\end{align*}
Since we have
$\| \partial^\alpha m_{j}  \|_{ H^{q^\prime}_s } \lesssim 2^{j([s]+1)}$ for $1 < q \leq 2$
by Lemma \ref{deriv mj},
we obtain
$\| p_{j,m} \|_{ H^{q^\prime}_s } \lesssim 2^{ - m ([s]+1) }$.
By the same manner as above, we also have 
$\| p_{j,m} \|_{ H^2_{\widetilde s} } \lesssim 2^{ - m ([s]+1) }$ for $2 < q \leq \infty$.
\end{proof}

We are now ready to prove Proposition \ref{thm on fou reduction}.

\begin{proof}[\bf Proof of Proposition \ref{thm on fou reduction}]
We recall the alternative form of $G(f)$ given in \eqref{final decom f}, that is,
\begin{equation*}
G(f) = G(S_0 f) 
+ \sum_{j = 0}^\infty q_j \cdot \Delta_j f 
+ \sum_{j = 0}^\infty \sum_{m = 0}^\infty p_{j,m} \cdot \Delta_j f ,
\end{equation*}
and prove that the function $G(f)$ belongs to $\mathcal{F} L^q_s$, 
which will be archived by three steps. 
In the first and second steps, we consider the second and third summations, 
and then consider $G(S_0 f)$ in the last step.

\noindent{\bf Step 1:}
We first consider the case $q < \infty$.
Taking the $\mathcal{F} L^q_s$-norm to the second summation in \eqref{final decom f}, we have
\begin{equation} \label{step1-1}
\left\| \langle \cdot \rangle^s \sum_{j = 0}^\infty 
\mathcal{F} \left[ q_j \cdot \Delta_j f \right]
\right\|_{L^q}
= 
\left( \sum_{\ell = 0}^\infty \int_{\Omega_\ell} \langle \xi \rangle^{sq} \left| \sum_{j = 0}^\infty \mathcal{F} \left[ q_j \cdot \Delta_j f \right] (\xi) \right|^q d\xi\right)^{1/q},
\end{equation}
where $\Omega_\ell = \{ \xi : 2^\ell < | \xi | \leq 2^{\ell+1} \}$ if $\ell \neq 0$ 
and $\Omega_0 = \{ \xi : | \xi | \leq 2 \}$.
We remark that 
\begin{equation*}
\supp \mathcal{F} \left[ q_j \cdot \Delta_j f  \right] \subset \{ \xi : | \xi | \leq C \cdot 2^{j+1} \},
\end{equation*}
since $\mathcal{F} \left[ q_j \cdot \Delta_j f  \right] = \left[ \varphi (\frac{\cdot}{C \cdot2^j}) \widehat {m_{j}} \right] \ast [ \psi_j \widehat f ] $. This means that on the domain $\Omega_\ell$, $\mathcal{F} \left[ q_j \cdot \Delta_j f  \right]$ always vanishes unless $j \geq \ell - N$ ($j \geq 0$ if $\ell = 0, \cdots, N$), 
where $N$ is a constant which depends only on $C\gg1$ (roughly, $2^N \sim C$). 
Hence, the right hand side of \eqref{step1-1} is equal to
\begin{equation} \label{step1-2}
\left( \sum_{\ell = 0}^\infty \int_{\Omega_\ell} \langle \xi \rangle^{sq} \left| \sum_{j = \ell -N}^\infty \mathcal{F} \left[ q_j \cdot \Delta_j f \right] (\xi) \right|^q d\xi\right)^{1/q},
\end{equation}
where the inner summation should be read as $\sum_{j = 0}^\infty$ if $\ell = 0,\cdots,N$. 
Then, using the H\"older inequality to the inner summation, we have
\begin{align}
\label{step1-3}
\begin{split}
\eqref{step1-2}
&\lesssim 
\left( \sum_{\ell = 0}^\infty \int_{\Omega_\ell} 2^{\ell sq} \left( \sum_{j = \ell - N}^\infty 2^{jsq} \left| \mathcal{F} \left[ q_j \cdot \Delta_j f \right] (\xi) \right|^q \right) \cdot \left( \sum_{j = \ell - N}^\infty 2^{-jsq^\prime} \right)^{q/q^\prime} d\xi\right)^{1/q}
\\
&\lesssim
\left( \sum_{\ell = 0}^\infty \int_{\Omega_\ell} \sum_{j = \ell -N}^\infty 2^{jsq} \left| \mathcal{F} \left[ q_j \cdot \Delta_j f \right] (\xi) \right|^q d\xi\right)^{1/q}
\\
&\lesssim
\left(\sum_{j = 0}^\infty 2^{jsq} \int_{\mathbb{R}^n} \left| \mathcal{F} \left[ q_j \cdot \Delta_j f \right] (\xi) \right|^q d\xi\right)^{1/q}
.
\end{split}
\end{align}
Here, in the last inequality, we used the fact that $\mathbb{R}^n = \bigcup_{\ell = 0}^\infty \Omega_\ell$.
Now, we observe that 
\begin{equation*}
\| \mathcal{F} \left[ q_j \cdot \Delta_j f \right] \|_{L^q} 
= \left\| \widetilde{q_j} ( D ) \left[ \psi_j \cdot \widehat f  \right] \right\|_{L^q},
\end{equation*}
where $\widetilde{q_j} ( x ) = q_j (- x)$. 
Then, we see that the last quantity of \eqref{step1-3} is equal to
\begin{equation} \label{step1-4}
\left(\sum_{j = 0}^\infty 2^{jsq} \left\| \widetilde{q_j} ( D ) \left[ \psi_j \cdot \widehat f  \right] \right\|_{L^q} ^q\right)^{1/q}.
\end{equation}
Apply to \eqref{step1-4} Proposition \ref{hahn} with $p = q^\prime$ for $4/3 \leq q \leq 2$ 
and with $p=2$ for $2 < q < \infty$.
Here, we note that the assumption $4/3 \leq q \leq 2$
is used to assure the conditions $2q^\prime/(q^\prime+2) \leq q \leq 2q^\prime/(q^\prime-2)$ and $q^\prime \geq 2$ in Proposition \ref{hahn}.
Then, we have
\begin{equation*} \label{step1-5}
\eqref{step1-4}
\lesssim
\left\{
\begin{array}{ll}
\displaystyle{
\left(\sum_{j = 0}^\infty 2^{jsq} \left\| q_j \right\|_{ H^{q^\prime}_s }^q \left\| \psi_j \cdot \widehat f  \right\|_{L^q}^q\right)^{1/q}
} 
& 
\quad\textrm{if}\quad
4/3 \leq q \leq 2
; 
\\
\displaystyle{
\left(\sum_{j = 0}^\infty 2^{jsq} \left\| q_j \right\|_{ H^2_{\widetilde s} }^q \left\| \psi_j \cdot \widehat f  \right\|_{L^q}^q\right)^{1/q}
} 
& 
\quad\textrm{if}\quad
2 < q < \infty,
\end{array}
\right.
\end{equation*}
where $\widetilde s$ is the number satisfying that $ n/2 < \widetilde s < n/2 + (s - n/q^\prime)$. 
Thus, we obtain from Lemma \ref{lem qk}
\begin{equation*} \label{step1-5-2}
\eqref{step1-4}
\lesssim
\left(\sum_{j = 0}^\infty 2^{jsq} \left\| \psi_j \cdot \widehat f  \right\|_{L^q}^q\right)^{1/q} 
.
\end{equation*}
Since it follows that $\sum_{j = 0}^\infty | \psi_j |^q \lesssim 1$ (if $q < \infty$)
and $2^{j} \sim \langle \xi \rangle$ on the support of $\psi_j$,
we realize that
\begin{align*}
\left(\sum_{j = 0}^\infty 2^{jsq} \left\| \psi_j \cdot \widehat f  \right\|_{L^q}^q\right)^{1/q} 
\sim
\left(\sum_{j = 0}^\infty \left\| \psi_j \cdot \langle \xi \rangle^s \widehat f  \right\|_{L^q}^q\right)^{1/q} 
\lesssim
\| f \|_{\mathcal{F} L^q_s }
\end{align*} 
for $4/3 \leq q <\infty$, which gives the desired result for the case $4/3 \leq q <\infty$.

We next consider the case $q = \infty$. 
However, this case is obtained similarly to the above. 
In fact, we have
\begin{equation}
\label{q infty sup}
\sup_{\xi \in \mathbb{R}^n } 
\left|
\langle \xi \rangle^s \sum_{j = 0}^\infty 
\mathcal{F} \left[ q_j \cdot \Delta_j f \right] (\xi)
\right|
\lesssim
\sup_{\ell \in \mathbb{Z}_+} 
\left(
\sup_{\xi \in \Omega_\ell } 2^{\ell s}
\sum_{j = \ell -N}^\infty 
\left|
\mathcal{F} \left[ q_j \cdot \Delta_j f \right] (\xi)
\right|
\right)
,
\end{equation}
since each $\Omega_\ell$ is disjoint.
Recalling from Lemma \ref{lem qk} that $\| q_j \|_{ H^2_{\widetilde s} } \lesssim 1$ 
holds independently of $j \in \mathbb{Z}_+$, 
we have by Remark \ref{rem hahn}
\begin{equation*}
\left|
\mathcal{F} \left[ q_j \cdot \Delta_j f \right] (\xi)
\right| 
\lesssim
\| q_j \|_{ H^2_{\widetilde s} } 
\left\| \psi_j \langle \cdot \rangle^{-s} \cdot \langle \cdot \rangle^s \widehat f \right\|_{L^\infty}
\lesssim
2^{-js} \left\| f \right\|_{\mathcal{F} L^\infty_s}.
\end{equation*}
Hence, we obtain 
\begin{equation*}
 2^{\ell s}
\sum_{j = \ell -N}^\infty 
\left|
\mathcal{F} \left[ q_j \cdot \Delta_j f \right] (\xi)
\right|
\lesssim
2^{\ell s}
\sum_{j = \ell -N}^\infty 
2^{-js} \| f \|_{ \mathcal{F} L^\infty_s } 
\lesssim
\| f \|_{ \mathcal{F} L^\infty_s } 
\end{equation*}
for any $\ell \in \mathbb{Z}_+$, 
where all the implicit constants above are independent of $\ell \in \mathbb{Z}_+$.
Substituting this estimate into \eqref{q infty sup}, we have the desired result for the case $q = \infty$.

Combining all the calculations above, we obtain for $4/3 \leq q \leq \infty$
\begin{equation} \label{step1 conc}
\left\| \sum_{j = 0}^\infty q_j \cdot \Delta_j f  \right\|_{\mathcal{F} L^q_s}
<
\infty.
\end{equation}

\noindent{\bf Step 2:}
We first consider the case $q < \infty$.
As in Step 1, we take the $\mathcal{F} L^q_s$-norm
to the third summation in \eqref{final decom f}
and decompose the $L^q$-norm by the dyadic decomposition. 
Then, we have
\begin{equation} \label{step2-1}
\left\| \langle \xi \rangle^s \sum_{j = 0}^\infty \mathcal{F} \left[ \sum_{m = 0}^\infty p_{j,m} \cdot \Delta_j f \right] \right\|_{L^q}
\lesssim
\sum_{m = 0}^\infty \left( \sum_{\ell = 0}^\infty \int_{\Omega_\ell} 2^{\ell sq} \left| \sum_{j = 0}^\infty \mathcal{F} \left[ p_{j,m} \cdot \Delta_j f \right] (\xi) \right|^q d\xi\right)^{1/q},
\end{equation}
where $\Omega_\ell = \{ \xi : 2^\ell < | \xi | \leq 2^{\ell+1} \}$ if $\ell \neq 0$ 
and $\Omega_0 = \{ \xi : | \xi | \leq 2 \}$.
Considering the support of $\mathcal{F} \left[ p_{j,m} \cdot \Delta_j f \right] $,
since we have 
\begin{equation*}
\mathcal{F} p_{j,m} \subset \{ \xi : C \cdot 2^{j+m-1} \leq | \xi | \leq C \cdot 2^{j+m+1} \} 
\quad\textrm{and}\quad 
\mathcal{F} \left[ \Delta_j f \right] \subset \{ \xi : 2^{j-1} \leq | \xi | \leq 2^{j+1} \},
\end{equation*}
we see that
\begin{equation*}
\supp \mathcal{F} \left[ p_{j,m} \cdot \Delta_j f \right]  
\subset \{ \xi : C \cdot 2^{j+m-2} \leq | \xi | \leq C \cdot 2^{j+m+2}\}.
\end{equation*}
This implies that on the domain $\Omega_\ell$, 
the function $\mathcal{F} \left[ p_{j,m} \cdot \Delta_j f \right]$ always vanishes 
unless $j,\ell,m \in \mathbb{Z}_+$ satisfy that $ j + m + N -2 \leq \ell \leq j + m + N + 1$,
where $N$ is the constant which depends only on $C\gg1$.
Put $\Lambda = \{ j \in \mathbb{Z}_+ : \ell - m - N - 1 \leq j \leq \ell - m - N + 2 \}$, 
where this set is read as $\Lambda = \varnothing$ if $\ell - m - N + 2 < 0$.
Then, $0 \leq \#\Lambda \leq 4$.
Hence, the right hand side of \eqref{step2-1} is equivalent to
\begin{align}
\label{step2-2}
\sum_{m = 0}^\infty 
\left( 
	\sum_{\ell = 0}^\infty 
		\int_{\Omega_\ell} 
			\sum_{j \in \Lambda} 2^{(j+m)sq} 
			\left| \mathcal{F} \left[ p_{j,m} \cdot \Delta_j f \right] (\xi) \right|^q 
		d\xi
\right)^{1/q} .
\end{align}
Then, we have by the Fubini-Tonelli theorem
\begin{align}
\label{step22}
\begin{split}
\eqref{step2-2}
&\leq
\sum_{m = 0}^\infty 2^{ms} \left( \int_{\mathbb{R}^n} \sum_{j =0}^\infty 2^{jsq} \left| \mathcal{F} \left[ p_{j,m} \cdot \Delta_j f \right] (\xi) \right|^q d\xi\right)^{1/q} 
\\
&=
\sum_{m = 0}^\infty 2^{ms} \left( \sum_{j =0}^\infty 2^{jsq} \left\| \mathcal{F} \left[ p_{j,m} \cdot \Delta_j f \right] \right\|_{L^q}^q \right)^{1/q}
.
\end{split}
\end{align}
Using the identity
$\| \mathcal{F} [ p_{j,m} \cdot \Delta_j f ] \|_{L^q} 
= 
\| \widetilde{ p_{j,m} } ( D ) [ \psi_j \cdot \widehat f ] \|_{L^q}
$,
where $\widetilde{ p_{j,m} } ( x ) = p_{j,m} (- x)$, 
we see that the last quantity of \eqref{step22} is equal to
\begin{align}
\label{step2-3}
\sum_{m = 0}^\infty 2^{ms} \left( \sum_{j = 0}^\infty 2^{jsq} \left\| \widetilde{ p_{j,m} } ( D ) \left[ \psi_j \cdot \widehat f  \right] \right\|_{L^q}^q \right)^{1/q}
.
\end{align}
As in Step 1, we have
\begin{equation*} \label{step1-5}
\eqref{step1-4}
\lesssim
\left\{
\begin{array}{ll}
\displaystyle{
\sum_{m = 0}^\infty 2^{ms} \left( \sum_{j = 0}^\infty 2^{jsq} \left\| p_{j,m} \right\|_{H^{q^\prime}_s}^q \left\| \psi_j \cdot \widehat f \right\|_{L^q}^q \right)^{1/q}
} 
& 
\quad\textrm{if}\quad
4/3 \leq q \leq 2
; 
\\
\displaystyle{
\sum_{m = 0}^\infty 2^{ms} \left( \sum_{j = 0}^\infty 2^{jsq} \left\| p_{j,m} \right\|_{ H^2_{\widetilde s} }^q \left\| \psi_j \cdot \widehat f \right\|_{L^q}^q \right)^{1/q}
} 
& 
\quad\textrm{if}\quad
2 < q < \infty,
\end{array}
\right.
\end{equation*}
for $ n/2 < \widetilde s < n/2 + (s - n/q^\prime)$.
Hence, recalling the properties that $\sum_{j = 0}^\infty | \psi_j |^q \lesssim 1$ (if $q < \infty$)
and $2^{j} \sim \langle \xi \rangle$ on $\supp \psi_j$,
we have by Lemma \ref{lem pkm} 
\begin{align*}
\eqref{step2-3}
\lesssim
\sum_{m = 0}^\infty 2^{ms} \cdot 2^{-m([s]+1)} \left( \sum_{j = 0}^\infty 2^{jsq}\left\| \psi_j \cdot \widehat f \right\|_{L^q}^q \right)^{1/q}
\lesssim
\| f \|_{\mathcal{F} L^q_s}
\end{align*}
for $4/3 \leq q < \infty$, which gives the desired result for the case $4/3 \leq q <\infty$.

We next consider the case $q = \infty$, which is obtained similarly to the above. 
In fact, we have
\begin{equation*}
\sup_{\xi \in \mathbb{R}^n} 
	\left| 
		\langle \xi \rangle^s \sum_{j = 0}^\infty \mathcal{F} 
			\left[ 
			\sum_{m = 0}^\infty p_{j,m} \cdot \Delta_j f 
			\right] (\xi) 
	\right|
\lesssim
\sup_{\ell \in \mathbb{Z}_+}
\left(
\sup_{\xi \in \Omega_\ell} 
\sum_{m = 0}^\infty 
	2^{\ell s} 
		\sum_{j \in \Lambda }
		\left| \mathcal{F} 
			\left[ p_{j,m} \cdot \Delta_j f \right] (\xi) 
		\right| 
\right)
\end{equation*}
(see above for the definition of the sets $\Omega_\ell$ and $\Lambda$).
Recalling from Lemma \ref{lem pkm} that 
$ \left\| p_{j,m} \right\|_{ H^{2}_{\widetilde s} } \lesssim 2^{-m([s]+1)}$ 
holds independently of $j,m \in \mathbb{Z}_+$ and following the same lines as in Step 1,
we have
\begin{equation*}
\left|
\mathcal{F} \left[ p_{j,m} \cdot \Delta_j f \right] (\xi)
\right| 
\lesssim
\|  p_{j,m}  \|_{ H^2_{\widetilde s} } 
\left\| \psi_j \widehat f \right\|_{L^\infty}
\lesssim
2^{-m([s]+1)} \cdot 2^{-js} \| f \|_{ \mathcal{F} L^\infty_s } 
.
\end{equation*}
Hence, we obtain
\begin{align*}
\sum_{m = 0}^\infty 2^{\ell s} 
	\sum_{j \in \Lambda }
		\left| \mathcal{F} 
			\left[ p_{j,m} \cdot \Delta_j f \right] (\xi) 
		\right| 
\lesssim
\sum_{m = 0}^\infty 
	2^{-m([s]+1)} \cdot 2^{\ell s} 
		\sum_{j \in \Lambda }
		2^{-js} \| f \|_{ \mathcal{F} L^\infty_s } 
\sim
\| f \|_{ \mathcal{F} L^\infty_s } 
\end{align*}
for any $\ell \in \mathbb{Z}_+$.
This gives the desired result for the case $q = \infty$.

Combining all the calculations above, we obtain for $4/3 \leq q \leq \infty$
\begin{equation} \label{step2 conc}
\left\|  \sum_{j = 0}^\infty \sum_{m = 0}^\infty p_{j,m} \cdot \Delta_j f  \right\|_{\mathcal{F} L^q_s}
<
\infty.
\end{equation}

\noindent{\bf Step 3:}
Lastly, we prove that $G(S_0 f) \in \mathcal{F} L^q_s$. 
Observe that
\begin{align*} 
G(S_0 f) 
= 
\int_0^1 G^{(1)} (t \cdot S_0 f) d t \cdot S_0 f
=
m_f \cdot S_0 f,
\end{align*}
where $m_f = \int_0^1 G^{(1)} (t \cdot S_0 f) d t$.
Then, since $\mathcal{F} L^q_r \hookrightarrow \mathcal{F} L^q_s$ for $r \geq s$ 
and $\langle \xi \rangle^r \lesssim 1 + | \xi_1 |^{r} \cdots + | \xi_n |^{r}$ for $r \geq 0$, 
we have by Proposition \ref{hahn} for $4/3 \leq q \leq 2$
\begin{align*}
\left\| \langle \xi \rangle^s \mathcal{F} \left[ G(S_0 f) \right] \right\|_{ L^q } 
&\lesssim
\left\| \mathcal{F} \left[ m_f \cdot S_0 f \right] \right\|_{ L^q } 
+
\sum_{ \ell = 1 }^n 
\left\| \mathcal{F} \left[ \partial_\ell^{[s]+1}  \left( m_f \cdot S_0 f \right) \right] \right\|_{ L^q } 
\\
&\lesssim
\left\| \widetilde{m_f} (D) \left[ \varphi \cdot \widehat f \right] \right\|_{ L^q } 
+
\sum_{ \ell = 1 }^n \sum_{\mu =0}^{[s]+1}
\left\| \widetilde{ \partial_\ell^{\mu} m_f } (D) \left[ \xi_\ell^{ [s]+1 - \mu }\varphi \cdot \widehat f \right] \right\|_{ L^q } 
\\
&\lesssim 
\left\| m_f \right\|_{H_s^{q^\prime}}
\left\| \varphi \cdot \widehat f \right\|_{ L^q } 
+
\sum_{ \ell = 1 }^n \sum_{\mu =0}^{[s]+1}
\left\| \partial_\ell^{\mu} m_f \right\|_{H_s^{q^\prime}}
\left\| \xi_\ell^{ [s]+1 - \mu }\varphi \cdot \widehat f \right\|_{ L^q }
\\
&\lesssim 
\left\| f \right\|_{ \mathcal{F} L^q_s } 
\sum_{ \ell = 1 }^n \sum_{\mu =0}^{[s]+1}
\left\| \partial_\ell^{\mu} m_f \right\|_{H_s^{q^\prime}},
\end{align*}
and have for $2 < q \leq \infty$
\begin{align*}
\left\| \langle \xi \rangle^s \mathcal{F} \left[ G(S_0 f) \right] \right\|_{ L^q } 
\lesssim
\left\| f \right\|_{ \mathcal{F} L^q_s }
\sum_{ \ell = 1 }^n \sum_{\mu =0}^{[s]+1}
\left\| \partial_\ell^{\mu} m_f \right\|_{H^2_{\widetilde s}},
\end{align*}
where we used the notation $ \widetilde{ h } (\xi) = h (- \xi)$.
Moreover, as in the proof of Lemma \ref{deriv mj}, 
Theorem A yields that for $\mu = 0,1, \cdots, [s]+1$
\begin{align*}
\left\| \partial_\ell^{\mu} m_f \right\|_{H_s^{q^\prime}}
&\lesssim
\sum_{ \nu = 0 }^{ \mu }
\left\| f \right\|_{ \mathcal{F}L^q_s }^\nu
<
\infty 
\quad\textrm{if}\quad
1 < q \leq 2; 
\\
\left\| \partial_\ell^{\mu} m_f \right\|_{H^2_{\widetilde s}}
&\lesssim
\sum_{ \nu = 0 }^{ \mu }
\left\| f \right\|_{ \mathcal{F}L^q_s }^\nu
<
\infty 
\quad\textrm{if}\quad
2 < q \leq \infty,
\end{align*}
since the assumption $f \in \mathcal{F} L^q_s$ with $s > n/q^\prime$ gives that
for $\beta \in \mathbb{Z}^n_+$,
$
\| \partial^{\beta} (S_0 f ) \|_{ H^{q^\prime}_s } 
\lesssim
\| f \|_{ \mathcal{F} L^q_s }
$
if $1 < q \leq 2$, and
$ 
\| \partial^{\beta} (S_0 f ) \|_{ H^2_{\widetilde s} } 
\lesssim
\| f \|_{ \mathcal{F} L^q_s }
$
if $2 < q \leq \infty$.
Hence, we obtain $\| G(S_0 f) \|_{ \mathcal{F} L^q_s } < \infty$.

By Steps 1-3, we conclude that $G(f) \in \mathcal{F} L^q_s$ if $f \in \mathcal{F} L^q_s$.
\end{proof}

Now, all the preparations are completed, so that we give the proof of Theorem \ref{thm on fou}.

\begin{proof}[\bf Proof of Theorem \ref{thm on fou}] 
As is stated at the beginning of this section, 
$F(f)$ with $F \in C^\infty (\mathbb{R})$ and $F(0)=0$ is given by
\begin{align*}
F(f) = G(f) + \sum_{k = 1}^N F^{ (k) } (0) \frac{ f^k }{ k! }
\end{align*}
for any $N \geq 0$, 
where $G \in C^\infty (\mathbb{R})$ and $G (0) = G^{ (1) } (0)= \cdots= G^{ (N) } (0) = 0$.
Choosing $N = [s]+2$, 
we obtain from Proposition \ref{thm on fou reduction}
that $G(f) \in \mathcal{F} L^q_s$ if $f \in \mathcal{F} L^q_s $.
The second one is shown by Proposition \ref{alge fou}. 
In fact, since $F \in C^\infty (\mathbb{R})$, we have $| F^{ (k) } (0) | \lesssim 1$, so that it follows that
\begin{align*}
\left\| \sum_{k = 1}^N F^{ (k) } (0) \frac{ f^k }{ k! } \right\|_{ \mathcal{F} L^q_s } 
\lesssim
 \sum_{k = 1}^N \left\| f \right\|_{ \mathcal{F} L^q_s }^k < \infty
\end{align*}
if $f \in \mathcal{F} L^q_s $. Hence, we obtain that $F(f) \in \mathcal{F} L^q_s$ if $f \in \mathcal{F} L^q_s $.
\end{proof}

\section{Proof of Theorem \ref{thm on mod}} \label{sec5}

As in Section \ref{sec4}, $F(f)$ is expressed in the following form:
\begin{align}
\label{Ff on mod}
F(f) = G(f) + \sum_{k = 1}^N F^{ (k) } (0) \frac{ f^k }{ k! },
\end{align}
for any $N \in \mathbb{N}$, where $G (0) = G^{ (1) } (0)= \cdots= G^{ (N) } (0) = 0$. 
Applying a Taylor expansion to $G$, we have
\begin{align}
\label{g on mod}
G(f) = f^N \cdot H(f),
\quad\textrm{where}\quad
H(f) = \frac{ 1 }{ (N-1)! } \int_0^1 ( 1 - \theta )^{ N - 1 } G^{ (N) } ( \theta f ) d \theta.
\end{align}
Note that $H \in C^\infty(\mathbb{R})$ and $H (0) = 0$.
Hence, in this section, we mainly prove that $G(f)$ in \eqref{g on mod} belongs 
to $M^{p,q}_s$ if $f \in M^{p,q}_s$. 
In order to prove this, we prepare the following lemma:

\begin{lemma}
\label{stft for g}
Let $4/3 \leq q \leq \infty$ and $s > n / q^\prime$,
and let $N$ be an arbitrary natural number.
Suppose that $G$ is the function in \eqref{g on mod},
$f \in M^{p,q}_s$
and real-valued functions $\phi, \widetilde \phi \in C^\infty_0 (\mathbb{R}^n)$ 
satisfy that $\widetilde \phi \equiv 1$ on $\supp \phi$.
Then, we have
\begin{align*}
\left\| \langle \xi \rangle^s V_{\phi} \left[ G(f) \right] (x,\xi) \right\|_{ L^q ( \mathbb{R}^n_\xi ) }
\lesssim
\left\| \langle \xi \rangle^s V_{\widetilde \phi} f (x,\xi) \right\|_{ L^q ( \mathbb{R}^n_\xi ) }^N
\end{align*}
for any $x \in \mathbb{R}^n$.
Here, the implicit constant is independent of $x \in \mathbb{R}^n$.
\end{lemma}

\begin{proof}[\bf Proof of Lemma \ref{stft for g}]
We first observe from \eqref{g on mod}
and the assumption $\widetilde \phi (\cdot - x) \equiv 1$ on $\supp \phi (\cdot - x)$ that
\begin{align*}
V_\phi \left[ G(f) \right] (x,\xi)
&=
\int_{\mathbb{R}^n} e^{-i \xi \cdot t} \ 
{ \phi (t-x) } \cdot G \left( \widetilde \phi (t - x) f (t) \right) dt
\\
&=
\int_{\mathbb{R}^n} e^{-i \xi \cdot t} \ 
{ \phi (t-x) } \cdot \left( \widetilde \phi (t - x) f (t) \right)^N \cdot H \left( \widetilde \phi (t - x) f (t) \right) dt
\\
&= 
\mathcal{F} \left[ 
{ \phi (\cdot-x) } \cdot \left( \widetilde \phi (\cdot - x) f \right)^N \cdot H \left( \widetilde \phi (\cdot - x) f \right)
\right] (\xi).
\end{align*}
Multiplying the weight $\langle \xi \rangle^s$ to the both sides
and taking the $L^q $-norm with respect to the $\xi$-variable,
we have by Proposition \ref{alge fou}
\begin{align*}
\left\| \langle \xi \rangle^s V_{\phi} \left[ G(f) \right] (x,\xi) \right\|_{ L^q ( \mathbb{R}^n_\xi ) }
&=
\left\| \langle \xi \rangle^s \mathcal{F} \left[ 
{ \phi (\cdot-x) } \cdot \left( \widetilde \phi (\cdot - x) f \right)^N \cdot H \left( \widetilde \phi (\cdot - x) f \right)
\right] (\xi) \right\|_{ L^q ( \mathbb{R}^n_\xi ) }
\\
&\lesssim
\left\| \phi (\cdot-x) \right\|_{ \mathcal{F} L^q_s } \cdot
\left\| \widetilde \phi (\cdot - x) f \right\|_{ \mathcal{F} L^q_s }^N \cdot
\left\| H \left( \widetilde \phi (\cdot - x) f \right) \right\|_{ \mathcal{F} L^q_s }.
\end{align*}
It obviously follows that $\left\| \phi (\cdot-x) \right\|_{ \mathcal{F} L^q_s }  \sim 1$ and
$
\| \widetilde \phi (\cdot - x) f \|_{ \mathcal{F} L^q_s } 
=
\| \langle \xi \rangle^s V_{\widetilde \phi} f (x,\xi) \|_{ L^q ( \mathbb{R}^n_\xi ) }
$.
We only consider $\| H ( \widetilde \phi (\cdot - x) f ) \|_{ \mathcal{F} L^q_s }$ to obtain the conclusion.
By Lemma \ref{mod sim fou} and Proposition \ref{alge mod}, we have
\begin{align*}
\left\| \widetilde \phi (\cdot - x) f \right\|_{ \mathcal{F} L^q_s }
\sim
\left\| \widetilde \phi (\cdot - x) f \right\|_{ M_s^{p,q} }
\lesssim
\| \widetilde \phi \|_{ M_s^{p,q} } \cdot
\left\| f \right\|_{ M_s^{p,q} }
<
\infty,
\end{align*}
where the implicit constants are both independent of $x \in \mathbb{R}^n$.
Then, recalling that $H\in C^\infty(\mathbb{R})$ and $H(0) = 0$,
we have $\sup_{x\in\mathbb{R}^n} \| H ( \widetilde \phi (\cdot - x) f ) \|_{ \mathcal{F} L^q_s } < \infty$
by Theorem \ref{thm on fou} if $4/3 \leq q \leq \infty$.
Hence, we obtain
\begin{align*}
\left\| \langle \xi \rangle^s V_{\phi} \left[ G(f) \right] (x,\xi) \right\|_{ L^q ( \mathbb{R}^n_\xi ) }
\lesssim
\left\| \langle \xi \rangle^s V_{\widetilde \phi} f (x,\xi) \right\|_{ L^q ( \mathbb{R}^n_\xi ) }^N.
\end{align*}
Here, recalling all the proofs in Section \ref{sec4}, 
we see that $\| H ( \widetilde \phi (\cdot - x) f ) \|_{ \mathcal{F} L^q_s }$ can be estimated 
by a polynomial of 
$\| \langle \xi \rangle^s V_{\widetilde \phi} f (x,\xi) \|_{ L^q ( \mathbb{R}^n_\xi ) }$.
This implies that the explicit order of the power in the right hand side can be actually taken larger than $N$. 
However, the explicit expression is not important,
since it is sufficient to understand that the order can be chosen arbitrarily large as we want.
Hence, we here omitted the details.
\end{proof}

Now, we are in a position to prove Theorem \ref{thm on mod}.

\begin{proof}[\bf Proof of Theorem \ref{thm on mod}]
We recall the expressions \eqref{Ff on mod} and have by Proposition \ref{alge mod}
\begin{align}
\label{Ff alge}
\| F(f) \|_{ M^{p,q}_s }
\lesssim
\| G(f) \|_{ M^{p,q}_s }
+
\sum_{k = 1}^N \| f \|_{ M^{p,q}_s }^k .
\end{align}
Here, we choose $N \in \mathbb{N}$ such that $N \geq [ \max(p/q, q/p) ] + 1$,
and it should be remarked that we exclude the cases
$p=\infty$ and $q < \infty$, or $p < \infty$ and $q=\infty$
in Theorem \ref{thm on mod},
since such $N$ cannot be taken in those cases.

We first consider $\| G(f) \|_{ M^{p,q}_s }$ for the case $p \leq q$.
Let real-valued functions $\phi, \widetilde \phi \in C^\infty_0 (\mathbb{R}^n)$ 
satisfy that $\widetilde \phi \equiv 1$ on $\supp \phi$.
Then, we have by the Minkowski inequality for integrals and Lemma \ref{stft for g}
\begin{align*}
\left\| G(f) \right\|_{ M^{p,q}_s }
&\lesssim
\left\| \left\| \langle \xi \rangle^s V_\phi \left[ G(f) \right] (x,\xi) \right\|_{ L^q (\mathbb{R}^n_\xi) } \right\|_{ L^p (\mathbb{R}^n_x) }
\\
&\lesssim
\left\| \left\| \langle \xi \rangle^s V_{\widetilde \phi} f (x,\xi) \right\|_{ L^q (\mathbb{R}^n_\xi) }^N \right\|_{ L^p (\mathbb{R}^n_x) }
=
\left\| \left\| \langle \xi \rangle^s V_{\widetilde \phi} f (x,\xi) \right\|_{ L^q (\mathbb{R}^n_\xi) } \right\|_{ L^{Np} (\mathbb{R}^n_x) }^N.
\end{align*}
Since $Np > q \geq p$, 
we have by Proposition \ref{basic emb}
\begin{align*}
\left\| G(f) \right\|_{ M^{p,q}_s }
\lesssim
\left\| \left\| \langle \xi \rangle^s V_{\widetilde \phi} f (x,\xi) \right\|_{ L^{Np} (\mathbb{R}^n_x) } \right\|_{ L^q (\mathbb{R}^n_\xi) }^N
\lesssim
\left\| \left\| \langle \xi \rangle^s V_{\widetilde \phi} f (x,\xi) \right\|_{ L^{p} (\mathbb{R}^n_x) } \right\|_{ L^q (\mathbb{R}^n_\xi) }^N
\sim
\left\| f \right\|_{ M^{p,q}_s }^N.
\end{align*}

We next assume that $q < p < \infty$. 
As above, Proposition \ref{basic emb} and Lemma \ref{stft for g} yield that
\begin{align*}
\left\| G(f) \right\|_{ M^{p,q}_s }
\lesssim
\left\| G(f) \right\|_{ M^{q,q}_s }
\lesssim
\left\| \left\| \langle \xi \rangle^s V_{\widetilde \phi} f (x,\xi) \right\|_{ L^q (\mathbb{R}^n_\xi) } \right\|_{ L^{Nq} (\mathbb{R}^n_x) }^N.
\end{align*}
Since $Nq > p > q$, we use Proposition \ref{basic emb} again and obtain
\begin{align*}
\left\| G(f) \right\|_{ M^{p,q}_s }
\lesssim
\left\| \left\| \langle \xi \rangle^s V_{\widetilde \phi} f (x,\xi) \right\|_{ L^{p} (\mathbb{R}^n_x) } \right\|_{ L^q (\mathbb{R}^n_\xi) } ^N 
\sim
\left\| f \right\|_{ M^{p,q}_s }^N.
\end{align*}
Therefore, for $1 \leq p < \infty$ and $4/3 \leq q < \infty$ (or $p = q = \infty$), 
we have $\| G(f) \|_{ M^{p,q}_s } \lesssim \| f \|_{ M^{p,q}_s }^N$.

Collecting all the estimates above, we obtain $\| F(f) \|_{ M^{p,q}_s } < \infty$. 
This is the desired conclusion.
\end{proof}

\appendix
\section{Local equivalence between modulation and Fourier Lebesgue spaces}

In this section, we state
that modulation spaces are locally equivalent to Fourier Lebesgue spaces.
The corresponding result for $s = 0$ was already proved by Okoudjou \cite[Lemma 1]{okoudjou 2009}, 
and the weighted case is obtained by following the same argument. 
However, for the reader's convenience, we give a proof.

\begin{lemma}
\label{mod sim fou}
Let $1 \leq p,q \leq \infty$ and $s \in \mathbb{R}$. 
Suppose that $\chi \in \mathcal{S} (\mathbb{R}^n) \setminus \{0\}$ 
satisfies that $\supp \chi \subset \{ x : | x - x_0 | \leq R \}$. 
Then, we have
$\left\| \chi \cdot f \right\|_{ M^{p,q}_s } 
\sim \left\| \chi \cdot f \right\|_{ \mathcal{F} L^q_s }$.
Here, the implicit constant is independent of $x_0 \in \mathbb{R}^n$, but depends on $R>0$.
\end{lemma}

\begin{proof}[\bf Proof of Lemma \ref{mod sim fou}]
Put $f_\chi = \chi \cdot f$.
We first prove the $\lesssim$ part. 
Choose $\phi \in \mathcal{S} (\mathbb{R}^n) \setminus \{0\}$ satisfying 
that $\supp \phi \subset \{ x : |x| \leq R \}$.  
Then, we see that $V_\phi [ f_\chi ] (x,\xi) $ always vanishes 
unless $x \in \mathbb{R}^n$ satisfies that $| x-x_0 | \leq 2R$.
Using the identity $| V_\phi [ f_\chi ] (x,\xi) | = | \overline{ \widehat \phi ( D - \xi ) } f_\chi (x) |$,
we have by the H\"older and Hausdorff-Young inequalities
\begin{align*}
\left\| V_\phi [ f_\chi ] (x,\xi) \right\|_{ L^p ( \mathbb{R}^n_x )}
&=
\left\| \chi_{B_{2R} (x_0)} (x) \cdot V_\phi [ f_\chi ] (x,\xi) \right\|_{ L^p ( \mathbb{R}^n_x )}
\\
&\lesssim 
R^{n/p} \left\| \overline{ \widehat \phi ( D - \xi ) } f_\chi (\cdot) \right\|_{ L^\infty}
\lesssim
R^{n/p} \left\| \widehat \phi ( t - \xi ) \cdot \mathcal{F}[ f_\chi ](t)\right\|_{ L^1 ( \mathbb{R}^n_t )}.
\end{align*}
Multiplying the weight $\langle \xi \rangle^s$ to the both sides, 
using the inequality $\langle \xi \rangle^s \lesssim \langle t \rangle^s \langle t - \xi \rangle^{ | s | }$ 
and taking the $L^q $-norm with respect to the $\xi$-variable, we have by the Young inequality
\begin{align*}
\left\| \left\| \langle \xi \rangle^s V_\phi [ f_\chi ] (x,\xi) \right\|_{ L^p ( \mathbb{R}^n_x )} \right\|_{ L^q (\mathbb{R}^n_\xi)}
&\lesssim
R^{n/p} \left\| \left\| \langle t - \xi \rangle^{ | s | } \widehat \phi ( t - \xi ) 
\cdot \langle t \rangle^s \mathcal{F}[ f_\chi ](t)\right\|_{ L^1 ( \mathbb{R}^n_t )} \right\|_{ L^q (\mathbb{R}^n_\xi)}
\\
&=
R^{n/p} \left\|\left( \langle \cdot \rangle^{ | s | } \left| \widehat \phi \right| \right)
\ast \left( \langle \cdot \rangle^s \left| \mathcal{F}[ f_\chi ] \right| \right) \right\|_{ L^q }
\lesssim
R^{n/p} \left\| \langle \cdot \rangle^s \mathcal{F}[ f_\chi ] \right\|_{ L^q}.
\end{align*}

We next prove the $\gtrsim$ part.
Choose $\phi \in \mathcal{S} (\mathbb{R}^n)$ satisfying that $\supp \phi \equiv 1$ on $ \{ x : |x| \leq 2R \}$. 
Then, $\phi(\cdot-x) \equiv 1$ on $\supp \chi$ 
if $x \in \mathbb{R}^n$ satisfies that $| x-x_0 | \leq R$.
Hence, it follows that
\begin{align*}
R^{n/p} \left| \mathcal{F} [ f_\chi ] (\xi) \right|
&\sim
\left\| \chi_{B_{R} (x_0)} (x) \cdot \mathcal{F} [ f_\chi ] (\xi) \right\|_{ L^p (\mathbb{R}^n_x) }
\\
&=
\left\| \chi_{B_{R} (x_0)} (x) \cdot \int_{\mathbb{R}^n} e^{-i\xi \cdot t} \phi (t-x) \cdot \chi(t) f (t) dt \right\|_{ L^p (\mathbb{R}^n_x) }
\leq
\left\| V_\phi [ f_\chi ] (x,\xi) \right\|_{ L^p (\mathbb{R}^n_x) }.
\end{align*}
Multiplying the weight $\langle \xi \rangle^s$ to the both sides
and taking the $L^q $-norm with respect to the $\xi$-variable, we have
\begin{align*}
\left\| \langle \cdot \rangle^s \mathcal{F} [ f_\chi ] \right\|_{ L^q }
\lesssim
R^{-n/p} \left\| \left\| \langle \xi \rangle^s V_\phi [ f_\chi ] (x,\xi) \right\|_{ L^p (\mathbb{R}^n_x) } \right\|_{ L^q (\mathbb{R}^n_\xi) }.
\end{align*}

Therefore, recalling the property that the modulation space norm is independent of the choice of window functions, we obtain $\| f_\chi \|_{ M^{p,q}_s } \sim \|  f_\chi \|_{ \mathcal{F} L^q_s } $.
\end{proof}

\section{Conditions for modulation spaces and Fourier Lebesgue spaces to be multiplication algebras}

In this section, we first consider necessary and sufficient conditions
for modulation spaces to be multiplication algebras,
that is,
for the estimate
\begin{equation*}
\| f \cdot g \|_{M^{p,q}_s} \lesssim \| f \|_{M^{p,q}_s} \cdot \| g \|_{M^{p,q}_s},
\end{equation*}
to hold.
They are given as follows.

\begin{proposition}
\label{opt for multi mod}
Let $1 \leq p \leq \infty$, $1 < q \leq \infty$ and $s \in \mathbb{R}$.
Then, the modulation space $M^{p,q}_s (\mathbb{R}^n)$ is a multiplication algebra 
if and only if the condition $s > n/q^\prime$ is satisfied.
\end{proposition}

Actually, this proposition is immediately obtained from \cite[Theorem 1.5]{guo fan wu zhao 2017}.
In fact, in \cite{guo fan wu zhao 2017}, necessary and sufficient conditions for the more general estimate
\begin{equation*}
\| f \cdot g \|_{M^{p,q}_s} \lesssim \| f \|_{M^{p_1,q_1}_{s_1}} \cdot \| g \|_{M^{p_2,q_2}_{s_2}}
\end{equation*}
were established, so that Proposition \ref{opt for multi mod} is given
by setting $p=p_1=p_2$, $q=q_1=q_2$ and $s=s_1=s_2$.
(We remark that, although only the case $q>1$ is considered in Proposition \ref{opt for multi mod},
the whole case $q \geq 1$ is treated in \cite{guo fan wu zhao 2017}.)
However, for reader's convenience,
we give a proof of Proposition \ref{opt for multi mod}
where the following two lemmas are essential:

\begin{lemma}
[{\cite[Proposition 5.1]{guo fan wu zhao 2017}}]
\label{necessity for multi 1}
Let $1 \leq p,q \leq \infty$ and $s \in \mathbb{R}$.
Then, if the modulation space $M^{p,q}_s$ is a multiplication algebra,
we have
$\ell^q_s \ast \ell^q_s \hookrightarrow \ell^q_s$.
\end{lemma}

\begin{lemma}
\label{necessity for multi 2}
Let $1 < q \leq \infty$ and $s \in \mathbb{R}$.
Then, if $\ell^q_s \ast \ell^q_s \hookrightarrow \ell^q_s$ holds,
we have $s > n/q^\prime$.
\end{lemma}

\begin{proof}[Proof of Lemma \ref{necessity for multi 2}]
We assume towards a contradiction that $s \leq n/q^\prime$.
Since $q >1$, we can take $\varepsilon >0$ such that $1 - 1/q -\varepsilon > 0$.
For this $\varepsilon > 0$, we define the sequences
\begin{align*}
a_{k,N}
&=
\left\{
\begin{array}{ll}
\langle k \rangle^{-n/q - s} \left( C + \log \langle k \rangle \right)^{-1/q - \varepsilon}, & \IF | k | \leq N, \\
0, & \ow;
\end{array}
\right.
\\
b_{k,N}
&=
\left\{
\begin{array}{ll}
1, & \IF N \leq |k| \leq 5N, \\
0, & \ow,
\end{array}
\right.
\end{align*}
in $k\in\mathbb{Z}^n$, 
where $N>0$ is a sufficiently large integer
and $C>1$ is a suitable constant which depends only on the dimension $n$.

We first estimate each sequence on $\ell^q_s$. 
For the case $q < \infty$,
the spherical coordinate transform yields that
\begin{align*}
\left\| a_{k,N} \right\|_{\ell^q_s}^q
&=
\sum_{ |k| \leq N } \langle k \rangle^{-n} \left( C + \log \langle k \rangle \right)^{-1 - \varepsilon q}
\\
&\lesssim
\int_{ |x| \leq 2N } \langle x \rangle^{-n} \left( C^\prime + \log \langle x \rangle \right)^{-1 - \varepsilon q} dx
\lesssim
\int_0^{2N} (1+r)^{-1} \left( 1 + \log \left( 1+r \right) \right)^{-1 - \varepsilon q} dr.
\end{align*}
By the change of variable $t = 1 + \log (1+r)$, we have
\begin{equation*}
\left\| a_{k,N} \right\|_{\ell^q_s}^q 
\lesssim
\int_1^{1+\log(1+2N)} t^{-1-\varepsilon q} dt \lesssim 1.
\end{equation*}
For the case $q = \infty$,  
we have $\| a_{k,N} \|_{\ell^\infty_s} \leq 1$, since $\varepsilon > 0$.
On the other hand, we have $\left\| b_{k,N} \right\|_{\ell^q_s} \sim N^{s+n/q}$ holds for $1 < q \leq \infty$.

Next, we consider the convolution $\{ a_{\cdot,N} \ast b_{\cdot,N} \}_{k \in \mathbb{Z}^n}$.
For $2N \leq |k| \leq 4N$, we have
\begin{align*}
\sum_{\ell \in \mathbb{Z}^n} a_{\ell,N} b_{k-\ell,N}
=
\sum_{N \leq |k-\ell| \leq 5N} a_{\ell,N}
=
\sum_{|\ell| \leq N} a_{\ell,N},
\end{align*}
since 
$\{ \ell \in \mathbb{Z}^n : |\ell| \leq N \} \subset \{ \ell \in \mathbb{Z}^n : N \leq |k-\ell| \leq 5N \}$
and $a_{\ell, N}=0$ if $|\ell| > N$.
Then by $s \leq n/q^\prime$ we have
\begin{align*}
\sum_{\ell \in \mathbb{Z}^n} a_{\ell,N} b_{k-\ell,N}
&=
\sum_{|\ell| \leq N} \langle \ell \rangle^{-n/q-s} \left( C + \log \langle \ell \rangle \right)^{-1/q-\varepsilon}
\\
&\geq
\sum_{|\ell| \leq N} \langle \ell \rangle^{-n} \left( C + \log \langle \ell \rangle \right)^{-1/q-\varepsilon}
\\
&\gtrsim
\int_{|x| \leq N/2} \langle x \rangle^{-n} \left( C^\prime + \log \langle x \rangle \right)^{-1/q-\varepsilon} dx
\\
&\sim
\int_0^{N/2} r^{n-1} (1+r)^{-n} \left( 1 + \log \left( 1+r \right) \right)^{-1/q-\varepsilon} dr
\\
&\gtrsim
\int_1^{N/2} (1+r)^{-1} \left( 1 + \log \left( 1+r \right) \right)^{-1/q-\varepsilon} dr,
\end{align*}
and hence by the same change of variable as above we have
\begin{align*}
\sum_{\ell \in \mathbb{Z}^n} a_{\ell,N} b_{k-\ell,N}
\gtrsim
\int_{1+\log2}^{1+\log (1+N/2)} t^{-1/q-\varepsilon}
\gtrsim
\left( 1 + \log \left( 1+N/2 \right) \right)^{1-1/q-\varepsilon}.
\end{align*}
This concludes that
\begin{equation*}
\left\| \{ a_{\cdot,N} \ast b_{\cdot,N} \}_{k \in \mathbb{Z}^n} \right\|_{\ell^q_s}
\geq
\left\| \{ a_{\cdot,N} \ast b_{\cdot,N} \}_{ k \in \{ 2N \leq |k| \leq 4N \} } \right\|_{\ell^q_s}
\gtrsim
N^{s+n/q}\left( 1 + \log \left( 1+N/2 \right) \right)^{1-1/q-\varepsilon}.
\end{equation*}

Collecting the estimates above, we have by the assumption $\ell^q_s \ast \ell^q_s \hookrightarrow \ell^q_s$
\begin{align*}
\| \{ a_{\cdot,N} \ast b_{\cdot,N} \}_{k \in \mathbb{Z}^n} \|_{\ell^q_s} 
\lesssim
\| a_{k,N} \|_{\ell^q_s} \cdot \| b_{k,N} \|_{\ell^q_s}
&\Longrightarrow
N^{s+n/q}\left( 1 + \log \left( 1+N/2 \right) \right)^{1-1/q-\varepsilon}
\lesssim
1 \cdot N^{s+n/q}
\\
&\Longleftrightarrow
\left( 1 + \log \left( 1+N/2 \right) \right)^{1-1/q-\varepsilon}
\lesssim
1.
\end{align*}
However, the last estimate fails when we choose a sufficiently large number $N>0$,
since $1 - 1/q -\varepsilon > 0$.
This contradicts to the assumption $\ell^q_s \ast \ell^q_s \hookrightarrow \ell^q_s$. 
Therefore, we obtain $s > n/q^\prime$.
\end{proof}

\begin{proof}[Proof of Proposition \ref{opt for multi mod}]
The ``IF'' part is given by Proposition \ref{alge mod}, and
the ``ONLY IF'' part is an immediate conclusion of 
Lemmas \ref{necessity for multi 1} and \ref{necessity for multi 2}.
\end{proof}

We also have a similar optimality for Fourier Lebesgue spaces:

\begin{proposition}
\label{opt for multi fou}
Let $1 < q \leq \infty$ and $s \in \mathbb{R}$.
Then, the Fourier Lebesgue space $\mathcal{F}L^q_s (\mathbb{R}^n)$ is a multiplication algebra 
if and only if the condition $s > n/q^\prime$ is satisfied.
\end{proposition}

For the proof of Proposition \ref{opt for multi fou}, we use the following lemma instead of Lemma \ref{necessity for multi 1}:

\begin{lemma}
[{\cite[Proposition 4.1]{guo fan wu zhao 2017}}]
\label{necessity for multi 3}
Let $1 \leq q \leq \infty$ and $s \in \mathbb{R}$.
Then, if the estimate 
\begin{equation*}
\left\| \langle \cdot \rangle^s \left( f \ast g \right) \right\|_{L^q}
\lesssim
\left\| \langle \cdot \rangle^s f \right\|_{L^q}
\cdot
\left\| \langle \cdot \rangle^s g \right\|_{L^q}
\end{equation*} 
holds,
we have
$\ell^q_s \ast \ell^q_s \hookrightarrow \ell^q_s$.
\end{lemma}

\begin{proof}[Proof of Proposition \ref{opt for multi fou}]
The ``IF'' part is given by Proposition \ref{alge fou}.
The ``ONLY IF'' part is an immediate conclusion of 
Lemmas \ref{necessity for multi 2} and \ref{necessity for multi 3}
if we notice the equivalence
\begin{equation*}
\| f \cdot g \|_{ \mathcal{F} L^q_s } \lesssim \| f \|_{ \mathcal{F} L^q_s } \cdot \| g \|_{ \mathcal{F} L^q_s }
\Longleftrightarrow
\| \langle \cdot \rangle^s ( \widehat{f} \ast \widehat{g} )\|_{L^q}
\lesssim
\| \langle \cdot \rangle^s \widehat{f} \|_{L^q}
\cdot
\| \langle \cdot \rangle^s \widehat{g} \|_{L^q}.
\end{equation*}
\end{proof}

\section*{Acknowledgments}
The first author is supported by Grant-in-aid for JSPS Research Fellow (No.~17J00359).
The second author is partially supported by Grant-in-aid for Scientific Research from JSPS (No.~26287022 and No.~26610021).
 The third author is partially supported by Grant-in-aid for Scientific Research from JSPS (No.~16K05201).

\end{document}